\documentclass[11pt,oneside]{amsart}
\calclayout

\usepackage{amsmath,amssymb,amsfonts,amsthm,mathabx,color,graphicx}
\usepackage[all]{xy}

\usepackage[backend=biber,style=alphabetic,maxalphanames=4,url=false,eprint=false]{biblatex}
\usepackage[colorlinks=true]{hyperref}
\usepackage{xurl}
\hypersetup{breaklinks=true}

\newcommand{\R}{\mathbb R}
\newcommand{\Z}{\mathbb Z}

\newcommand{\neb}{\mathcal N}
\newcommand{\diam}{\mathrm{diam}}

\newcommand{\aut}{\mathrm{Aut}}
\newcommand{\stab}{\mathrm{Stab}}
\newcommand{\mcg}{\mathrm{MCG}}

\newcommand{\llangle}{\left\langle\!\left\langle}
\newcommand{\rrangle}{\right\rangle\!\right\rangle}
\newcommand*{\free}{\raisebox{-0.6ex}{\scalebox{2.5}{$\ast$}}}

\newcommand{\bY}{\ensuremath{\mathbf{Y}}}
\newcommand{\bT}{\ensuremath{\mathbf{T}}}
\newcommand{\tsh}[1]{\left\{\kern-.7ex\left\{#1\right\}\kern-.7ex\right\}}
\newcommand{\Tsh}[2]{\tsh{#2}_{#1}}
\newcommand{\ignore}[2]{\Tsh{#2}{#1}}

\newcommand{\X}{\mathcal X}

\newcommand{\C}{\mathcal C}
\newcommand{\T}{\mathcal T}
\newcommand{\A}{\mathcal A}
\newcommand{\M}{\mathcal M}
\newcommand{\G}{\mathcal G}
\newcommand{\PP}{\mathcal P}
\newcommand{\s}{\mathfrak S}
\newcommand{\g}{\mathfrak g}
\newcommand{\pr}{\mathbf P}
\newcommand{\F}{\mathbf F}
\newcommand{\E}{\mathbf E}
\newcommand{\rel}{\mathrm{Rel}}
\newcommand{\cont}{\mathrm{cont^{\perp}}}

\newcommand{\nest}{\sqsubseteq}
\newcommand{\propnest}{\sqsubsetneq}
\newcommand{\orth}{\perp}
\newcommand{\notorth}{\not\perp}
\newcommand{\trans}{\pitchfork}

\newtheorem{thm}{Theorem}[section]
\newtheorem{lem}[thm]{Lemma} 
\newtheorem{cor}[thm]{Corollary}
\newtheorem{prop}[thm]{Proposition}
\newtheorem{que}[thm]{Question}
\newtheorem*{thmVeechQT}{Theorem 7.8$'$}

\theoremstyle{definition}
\newtheorem{defn}[thm]{Definition}

\newtheorem*{notation}{Notation}

\theoremstyle{remark}
\newtheorem{rem}[thm]{Remark}

\newtheorem{rem*}{Remark}

\begin{document}

\title{Property QT of relatively hierarchically hyperbolic groups}
\author{Bingxue Tao}
\address{Department of Mathematics, Kyoto University, Kyoto 606-8502, Japan}
\email{tao.bingxue.s66@kyoto-u.jp}

\begin{abstract}
    Using the projection complex machinery, Bestvina--Bromberg--Fujiwara, Hagen--Petyt, and Han--Nguyen--Yang have proved that several classes of nonpositively curved groups admit equivariant quasi-isometric embeddings into finite products of quasi-trees, i.e. having property QT. In this paper, we unify and generalize these results by establishing a sufficient condition for relatively hierarchically hyperbolic groups to have property QT. 
    
    As applications, we show that a group has property QT if it is residually finite and belongs to one of the following classes of groups: admissible groups, hyperbolic-$2$-decomposable groups with no distorted elements, Artin groups of large and hyperbolic type, and $\pi_1$-extension groups of lattice Veech groups. We also introduce a slightly stronger version of property QT, called property QT$_0$, and show the invariance of property QT$_0$ under graph products. 
\end{abstract}


\maketitle

\section{Introduction}

Group actions on quasi-trees have been studied intensively in recent years. A \emph{quasi-tree} is a geodesic space quasi-isometric to a simplicial tree. We say that a finitely generated group $G$ has \emph{property QT} if $G$ acts on a finite product of quasi-trees (equipped with the $\ell^1$-metric) such that the orbit map is a quasi-isometric embedding. 
Such an embedding is called a \emph{QT embedding} of $G$. Since a quasi-tree has asymptotic dimension at most $1$, property QT is a stronger form of finite asymptotic dimension. Examples of groups with property QT include
\begin{itemize}
    \item Coxeter groups \cite{DJ99};
    \item residually finite hyperbolic groups \cite{BBF21};
    \item mapping class groups of finite-type surfaces \cite{BBF21};
    \item virtually colorable hierarchically hyperbolic groups whose associated hyperbolic spaces are all quasi-trees \cite{HP22} (including virtually compact special groups \cite{BHS17b}, the genus $2$ handlebody group \cite{Che22}, fundamental groups of non-geometric graph manifolds \cite{HRSS25});
    \item fundamental groups of compact orientable $3$-manifolds whose sphere-disk decomposition does not support either Sol or Nil geometry \cite{HNY25};
\end{itemize}
along with their undistorted subgroups. 

The last four examples were proved to have property QT with the help of the projection complex techniques developed in \cite{BBF15,BBFS19}. In particular, property QT of mapping class groups strengthens \cite[Theorem C]{BBF15}, which says that mapping class groups equivariantly quasi-isometrically embed in a finite product of hyperbolic graphs of finite asymptotic dimension. Counterexamples of property QT include certain special linear groups \cite{Man06,Man08}, generalized Baumslag--Solitar groups with infinite monodromy \cite{But25} and groups with Property hereditary
(NL) \cite{BFG24}. For some basic corollaries of property QT, see \cite[\S 2.1\&2.2]{HNY25}. Recently, Vergara \cite{Ver24} proved that any finitely generated group with property QT has a proper uniformly Lipschitz affine action on $\ell^1$ with quasi-isometrically embedded orbits. 

As a generalization of the Masur–Minsky machinery \cite{MM99,MM00}, \emph{hierarchically hyperbolic groups} \cite{BHS17b,BHS19}, abbreviated as HHGs, have become an important bridge between mapping class groups, cubical groups, and many other nonpositively-curved groups. A list of papers in this field can be found in \cite{HRSS25}. Coarsely speaking, an HHG is a finitely generated group $G$ whose geometry can be recovered from $G$-equivariant projections to a specified (possibly infinite) collection of hyperbolic spaces. 
For background on HHGs and relative HHGs, see Section \ref{def2}. As shown in \cite{BHS17a}, HHGs have finite asymptotic dimension. This leads to a natural question:

\begin{que}
    Which HHGs have property QT?
\end{que}

In this paper, we provide a sufficient condition for relative HHGs to have property QT, which reproves those of \cite{BBF21}, \cite{HP22}, and \cite{HNY25} mentioned above. We also give a sufficient condition for the existence of a quasi-median QT embedding in the sense of \cite{HP22}. This stronger property can be used to prove the existence of globally stable cylinders (see \cite{PSZ25}), which connects to a long-standing question of Rips and Sela \cite{RS95} about canonical representatives of elements in hyperbolic groups. 
The following is a collection of applications from Section \ref{sec_apps}. These results are new except for mapping class groups. 

\begin{thm}\label{thm_apps}
    The following groups have property QT. 
    \begin{itemize}
        \item Mapping class groups of finite-type surfaces;
        \item Residually finite admissible graphs of groups;
        \item Residually finite hyperbolic-$2$-decomposable groups with no distorted elements;
        \item Residually finite Artin groups of large and hyperbolic type.
    \end{itemize}

    Moreover, the QT embeddings for all these groups are quasi-median. 
\end{thm}

\begin{rem*}
    In this updated arXiv version, we also show that every $\pi_1(\Sigma)$-extension group of a lattice Veech group in the mapping class group $\mcg(\Sigma)$ of a closed surface $\Sigma$ has property QT by answering Question \ref{que: VeechRF} from the published version (see Remark \ref{rem: VeechRF} and Theorem \hyperref[thm: VeechQT]{7.8$'$}). 
\end{rem*}

When studying the invariance of property QT in some cases, we want the group action on the product space to be diagonal. We say that a finitely generated group $G$ has \emph{property QT$_0$} if $G$ has property QT and the $G$-action on the finite product of quasi-trees is diagonal. By \cite[Theorem 1.5]{HNY25}, if a residually finite group $G$ is hyperbolic relative to a collection of groups with property QT$_0$, then $G$ has property QT. 
Without ambiguity, we also say that a $G$-action on a metric space $X$ has \emph{property QT$_0$} if $X$ admits a $G$-equivariant quasi-isometric embedding into a finite product of quasi-trees on which $G$ acts diagonally. In particular, if $X$ itself is a finite product of quasi-trees, then any diagonal action on $X$ has property QT$_0$. We prove the following invariance of property QT$_0$ under graph products. 

\begin{thm}\label{Introthm: QT_GraphProd}[Theorem \ref{thm_graphproduct}]
    Any graph product of groups whose every vertex group has property QT$_0$ still has property QT$_0$.
\end{thm}

Now we give the main definitions of this paper in order to state our main theorem.

\begin{defn}
\label{type}
    Let $(G,\s)$ be a relative HHG. For any $U\in \s$, We write $G_U<\aut(\s_U)$ to mean the image of $\stab_G(U)$ under the restriction homomorphism. 
    \begin{enumerate}
    \item We say a domain $U\in \s$ is of \emph{type I} if it has the following properties. 
    \begin{enumerate}
     \item (Hyperbolicity) $\C U$ is hyperbolic.
     \item (Acylindrical image) $G_U$ acts on $\C U$ acylindrically.
     \item (Cobounded nested region) $G_U$ acts on $\F_U$ coboundedly.
     \item (Separable quasi-axes) For any element $g\in\stab_G(U)$ that acts loxodromically on $\C U$, the elementary closure $EC(g)$ is separable in $G$. 
    \end{enumerate}

    \item We say a domain $U\in \s$ is of \emph{type II} if the action $G_U\curvearrowright \C U$ has property QT$_0$. 
    \end{enumerate}
\end{defn}

For any $U\in \s$ of type II, property QT$_0$ provides quasi-trees $T_U^i$ along with $G_U$-equivariant maps $\iota^i_U:\C U\to T_U^i$ for $i=1,\dots,n_U$ such that 
\[\prod_{i=1}^{n_U}\iota_U^i:\C U\to \prod_{i=1}^{n_U}T_U^i\]
is a quasi-isometric embedding. 
Our main theorem is as follows. 

\begin{thm}\label{maintheorem}
    Let $(G,\s)$ be a relative HHG that is virtually colorable. If every $U\in \s$ is of type I or type II, then $G$ has property QT. 

    Moreover, if for any $D\ge 1$, there exists $D'\ge 1$ such that for every $U\in \s$ of type II and each $i=1,\dots,n_U$, the map $\iota^i_U:\C U\to T_U^i$ sends $(D,D)$-quasi-geodesics to unparametrized $(D',D')$ quasi-geodesics, then $G$ is coarse median and the QT embedding of $G$ is quasi-median. 
\end{thm}

\begin{proof}[Sketch of proof]
    We roughly explain how to prove Theorem \ref{maintheorem} in the case that $(G,\s)$ is an HHG with only type I domains excluding the ``moreover'' part. This case contains most of the key ideas. 
    
    First, we introduce a class of ``thick'' distances on $G$ each of which is defined using a class of ``thick'' segments of hierarchy paths on $G$. We prove in Section \ref{sec_thickdistanceformula} a thick distance formula saying that the word metric of $G$ can be recovered by summing up these thick distances. This is an analogue of the distance formula for HHGs.

    Then we show that any class of thick segments is cofinite up to the group action. Furthermore, these thick segments can be extended to a cofinite collection of quasi-axes. Using projections to these quasi-axes, we can estimate the thick distance. This is done in Section \ref{sec_appoximation}

    Finally, we take a finite-index subgroup of $G$, say $H$, such that a collection of quasi-axes as above is divided into finitely many $H$-orbits. Each $H$-orbit satisfies the Bestvina--Bromberg--Fujiwara projection axioms so it gives us a quasi-tree with an $H$-action. We prove that $H$ equivariantly quasi-isometrically embeds in the product of these finitely many quasi-trees in Section \ref{sec_quasitrees}. Since property QT is commensurably invariant, $G$ has property QT.
\end{proof}

    As stated in the above proof, property QT is commensurably invariant \cite[\S 2.2]{BBF21}. It follows that the conclusion of Theorem \ref{maintheorem} also holds for any group that is virtually a relative HHG that satisfies our condition, even though such a group may not be a relative HHG itself \cite{PS23}. 

    For most examples of HHGs that emerged from the study, every domain satisfies the first three conditions of type I. Virtual colorability is also common in practice. Therefore, the biggest restriction of our theorem comes from the assumption of separable quasi-axes. We further discuss it in Section \ref{neatker} and show how residual finiteness helps to give an easy-to-use criterion for having separable quasi-axes.

\subsection*{Acknowledgements}
    The author is grateful to his PhD supervisor, Koji Fujiwara, for many helpful discussions on this paper. He thanks Wenyuan Yang for discussing with him the paper \cite{HNY25} that became the motivation for this work. He thanks Shengkui Ye and the anonymous referee for their helpful comments. He also thanks Zihao Liu for pointing out that Question \ref{que: VeechRF} has a simple answer. This work was supported by JST SPRING, Grant Number JPMJSP2110.

\section{Background}\label{background}
\label{sec_background}

\subsection{Quasi-isometric embeddings and acylindricity}
\label{back1}

Given constants $\lambda\ge 1$, $c\ge 0$, we say that a coarse map $f:X\to Y$ between metric spaces $(X,d_X)$ and $(Y,d_Y)$ is a $(\lambda,c)$-\emph{quasi-isometric embedding} if \[\frac{1}{\lambda}d_X(x_1,x_2)-c\le d_Y(f(x_1),f(x_2))\le\lambda d_X(x_1,x_2)+c\]
for all $x_1,x_2\in X$. A $(\lambda,c)$-quasi-isometric embedding $\gamma: [0,l] \to X$ is called a (parametrized) $(\lambda,c)$-\emph{quasi-geodesic} in $X$. A coarse map $\gamma:[0,l]\to X$ is an \emph{unparametrized} $(\lambda,c)$-quasi-geodesic if there is a strictly increasing function $f:[0,l]\to[0,l]$ with $f(0)=0$, $f(l)=l$ such that $\gamma \circ f$ is a $(\lambda,c)$-quasi-geodesic. We also use the term ``quasi-geodesic'' to mean a quasi-isometric embedding of $\R$. We will not distinguish between a quasi-geodesic and its image in $X$. 

A geodesic metric space is called $\delta$-\emph{hyperbolic} (or simply, \emph{hyperbolic}) for $\delta\ge 0$ if for any geodesics $\alpha,\beta,\gamma$ that form a triangle, $\alpha$ is contained in the $\delta$-neighborhood of $\beta\cup\gamma$ \cite{Gro87}. For a $\delta$-hyperbolic space $X$, an isometry $g:X\to X$ is called \emph{loxodromic} if the $g$-orbit $n\mapsto g^nx$ is a quasi-geodesic for some (equivalently, for any) $x\in X$. 

Let $X$ be a hyperbolic space and $G$ be a group acting by isometries on $X$ with a loxodromic element $g$. Given constants $\lambda\ge 1$, $c\ge 0$, a $(\lambda,c)$-quasi-geodesic $\gamma\subset X$ is called a $(\lambda,c)$-\emph{quasi-axis} for $g$ if $\gamma$ is $g$-invariant. The \emph{elementary closure} of $g$ in $G$, $EC_G(g)$, is the subgroup of $G$ that stabilizes $\gamma$ up to bounded Hausdorff distance. If there is no ambiguity in $G$, we often simplify the notation as $EC(g)$. Equivalently, it is the stabilizer of the set $\gamma(\pm\infty)$, the points at infinity of $\gamma$. Thus, the elementary closure does not depend on the choice of $\gamma$. Everything that commutes with $g$ is contained in $EC(g)$ (including powers and roots), but there may be other elements. 

A group action $G\curvearrowright X$ by isometries is called \emph{acylindrical} \cite{Bow08} if for any $r \ge 0$, there exist constants $R, N \ge 0$ such that for any pair $a, b \in X$ with $d(a, b) \ge R$, we have
\[
\# \bigl \{  g \in G \,|\, d(ga, a) \le r \,\, \textup{and} \,\, d(gb, b) \le r \bigr \} \le N.
\]

Let $X$ be a hyperbolic space and $G$ be a group acting acylindrically on $X$ with a loxodromic element $g$. Some basic properties of this kind of actions can be found in \cite{Osi16}. In this case, the elementary closure $EC(g)$ is the unique maximal virtually cyclic subgroup of $G$ that contains $g$ \cite[Lemma 6.5]{DGO17}. Moreover, $EC(g)$ has a subgroup of index at most $2$ that is a centralizer of a large power of $g$ in $G$ \cite[Corollary 6.6]{DGO17}. 

In this paper, we will consider group actions with a large kernel, in which case the action cannot be acylindrical. As in \cite{BBF21}, an action $G\curvearrowright X$ is said to have \emph{acylindrical image} if the image of $G$ in the isometry group of $X$ is acylindrical.

\subsection{Projection axioms}\label{projcpx}

In this section, we review the construction of a quasi-tree of spaces in \cite{BBF15} with improvements from \cite{BBFS19}. 

Let $\bY$ be a collection of geodesic metric spaces, and $\pi_Y(X)\subset Y$ be specified subsets whenever $X\neq Y$ are elements of $\bY$. Write $d^\pi_Y(X,Z)$ to mean $\diam(\pi_Y(X)\cup\pi_Y(Z))$ for $X\neq Y\neq Z$. We say that $(\bY, \{\pi_Y\})$ is a \emph{projection system} with \emph{projection constant} $\xi\ge 0$ if it satisfies the following \emph{projection axioms}. 
  \begin{enumerate}
  \item [(P0)] (Bounded projection) $\diam (\pi_Y(X))\leq\xi$ when $X\neq Y$. 
    \item [(P1)] (Behrstock inequality) if $X,Y,Z$ are distinct and $d^\pi_Y(X,Z)>\xi$, then
      $d^\pi_X(Y,Z)\leq\xi$. 
    \item [(P2)] (Finiteness) for $X\neq Z$ the set
      $$\{Y\in\bY\mid d^\pi_Y(X,Z)>\xi\}$$
      is finite. 
  \end{enumerate}
  
We say that $(\bY, \{\pi_Y\})$ is a $G$-\emph{projection system} if a group $G$ acts on the set $\bY$ in such a way that every $g\in G$ acts as an isometry from $Y$ to $gY$ and the projections $\pi_Y$ are $G$-equivariant, that is, $\pi_{gY}(gX)=g\pi_Y(X)$. 

If we replace (P1) with
\begin{enumerate}
\item [(P1)$'$] if $X,Y,Z$ are distinct and $d^{\pi}_Y(X,Z)>\xi$, then
      $\pi_X(Y)=\pi_X(Z)$,
\end{enumerate}
then we say that $(\bY, \{\pi_Y\})$ satisfies the {\em strong projection axioms}. While there are many natural situations where the projection axioms hold, the strong projection axioms are not as natural. However, we can modify the projections so that they do hold. 

\begin{thm}\cite[Theorem 4.1]{BBFS19}\label{strong axioms}
If $(\bY, \{\pi_Y\})$ is a projection system with constant $\xi$, then there are projections $\{\pi'_Y\}$ such that $(\bY, \{\pi'_Y\})$ satisfies the strong projections axioms with constant $\xi'$, where $\pi'_Y(X)$ and $\pi_Y(X)$ are apart from each other within a uniform Hausdorff distance $\epsilon$, and $\epsilon$ and $\xi'$ only depend on $\xi$. Moreover, if $(\bY, \{\pi_Y\})$ is a $G$-projection system, then $(\bY, \{\pi'_Y\})$ is still a $G$-projection system. 
\end{thm}

Let
$\C_K\bY$ denote the space obtained from the disjoint union
$$\bigsqcup_{Y\in\bY} Y$$
by joining points in $\pi_X(Z)$ with points in $\pi_Z(X)$ by an edge of length one whenever $d_Y(X,Z) <K$  for all $Y\in\bY-\{X,Z\}$. When the spaces are
graphs and projections are subgraphs, we can join just the vertices in
these projections. Moreover, if $\bY$ is a $G$-projection system, then $G$ acts isometrically on $\C_K\bY$. 

\begin{thm}\cite[\S 4]{BBF15}\label{bbf}
If $(\bY, \{\pi_Y\})$ satisfies the strong projection axioms with constant $\xi$, then for all $K>2\xi$ 
\begin{itemize}
 \item $\C_K\bY$ is hyperbolic if all $Y\in\bY$ are $\delta$-hyperbolic;
 \item $\C_K\bY$ is a quasi-tree if all $Y\in\bY$ are quasi-trees with uniform QI constants.
\end{itemize}
\end{thm}  

There is a very useful distance formula in $\C_K\bY$. Let $X,Z\in\bY$ and $x\in X$, $z\in Z$. We define
$\pi_Y(x)=\pi_Y(X)$ if $Y\ne X$ and define $\pi_X(x)=x$. Then define 
$d_Y(x,z)=\diam(\pi_Y(x)\cup \pi_Y(z))$.

\begin{notation}
Given $A,B\ge 0$, we define a threshold function by 
\[\ignore{A}{B} = \left\{\begin{array}{cl} A &\mbox{ if }A \ge B\\  0 & \mbox{ otherwise.}\end{array} \right.\]  
\end{notation}

\begin{thm}\label{distfor2}\cite[Theorem~6.3]{BBFS19}
 Let $(\bY, \{\pi_Y\})$ satisfy the strong projection axioms with constant $\xi$.
  Let $x\in X$ and $z\in Z$ be two points of $\C_K(\bY)$ with
  $X,Z\in\bY$. Then
  $$\frac 14 \sum_{Y\in\bY} \ignore{d_Y(x,z)}{K}\leq
  d_{\C_K\bY}(x,z)\leq 2\sum_{Y\in\bY} \ignore{d_Y(x,z)}{K}+3K$$
  for all $K \ge 4\xi$.
\end{thm}

Next we recall a theorem that allows us to pass projection axioms from a projection system to a collection of certain subspaces. 
Let $\bY$ be a collection of $\delta$-hyperbolic spaces and $(\bY, \{\pi_Y\})$ be a projection system with constant $\xi$. For each $Y\in\bY$, let $\A_Y$ be a collection of quasi-geodesics in $Y$. Let $\A$ be the disjoint union of all $\A_Y$'s. We also make the following assumptions. 

\begin{itemize}
\item As a collection of quasi-geodesics, $\A$ has uniform coarse constants. 
\item For $\alpha,\beta\in \A_Y$, we define $\pi_\alpha(\beta)$ to be the closest-point projection of $\beta$ to $\alpha$.
\item For $\alpha\in\A_X$ and $\beta\in\A_Y$ where $X\neq Y$, we define $\pi_\alpha(\beta)$ to be the closest point projection of $\pi_X(Y)$ to $\alpha$.
\end{itemize}

\begin{thm}\label{proj.axioms}\cite[Theorem~4.17]{BBF21}
For any $\theta>0$, there exists $\xi'>0$, depending only on $\theta,\delta,\xi$ and coarse constants of $\A$, such that the following holds. If $\diam (\pi_\alpha(\beta))\le \theta$ whenever $\alpha$ and $\beta$ are distinct elements in the same $\A_Y$, then $(\A, \{\pi_{\alpha}\})$ is a projection system with constant $\xi'$. 
\end{thm}

\subsection{(Relatively) hierarchically hyperbolic spaces}\label{def2}

In this paper, we deal with \emph{(relatively) hierarchically hyperbolic spaces} and \emph{(relatively) hierarchically hyperbolic groups}. Coarsely speaking, a (relative) HHS is a pair $(\X,\s)$, where $\X$ is a quasi-geodesic space and $\s$ is an index set, with some extra structure. A full definition can be found in \cite[Definition 1.1,1.21]{BHS19}. Some important information from the definition is collected below. 

\begin{itemize}
 \item An element $U\in \s$ is called a \emph{domain} of $\X$. $\s$ has a partial order $\nest$ (called \emph{nesting}) and a symmetric relation $\orth$ (called \emph{orthogonality}). These two relations are required to be mutually exclusive. Any two elements that are neither comparable under the partial order nor mutually orthogonal are defined to be mutually \emph{transversal} and we denote this relation by $\trans$. We denote the set of domains nested (respectively, properly nested) in $U$ by $\s_U$ (respectively, $\s_U^{\circ}$) . 

 \item There is a unique $\nest$-maximal element $S$ in $\s$ and a uniform bound on the length of $\nest$-chains in $\s$, called the \emph{complexity} of $(\X,\s)$. The \emph{level} $\ell(V)$ of $V\in\s$ is defined inductively as follows. If $V$ is $\nest$-minimal then $\ell(V)=1$. The element $V$ has level $k+1$ if $k$ is the maximal integer such that there exists $U\propnest V$ with $\ell(U)=k$. 

 \item For HHSes, there is a set $\{(\C U,d_U):U\in\s\}$ of uniformly hyperbolic spaces and a set of uniformly coarsely Lipschitz and coarsely surjective maps $\pi_U: \X\to\C U$ for all $U\in\s$. For relative HHSes, the complexity is at least $2$. If $U$ is $\nest$-minimal, $\C U$ is not required to be hyperbolic. This is the only difference between HHSes and relative HHSes in definition. 

 \item For $U\propnest V$ or $U\trans V$, there is a uniformly bounded set $\rho^U_V \subset \C V$. 
 \item For $U\propnest V$, there is a coarse map $\rho^V_U:\C V\to\C U$. 
 \item Whenever $V\nest W$ and $W\orth U$, we require that $V\orth U$. 
 \item (Orthogonal containers) For each $T\in\s$ and each $U\in\s_T$ for which $\{V\in\s_T\mid V\orth U\}\neq\emptyset$, there exists $W\in \s_T^{\circ}$, so that whenever $V\orth U$ and $V\nest T$, we have $V\nest W$. We say that $W$ is an \emph{orthogonal container} of $U$ in $T$ if $W$ is a $\nest$-minimal element satisfying the above property. Let $\text{cont}^{\perp}_T U$ denote the set of all orthogonal containers of $U$ in $T$. If $T$ is the maximal element of $\s$, then we suppress it from the notation and write $\cont U$. We set $\s_U^\orth=\{V\in\s\mid V\orth U\}\cup\{A\}$, where $A$ is an arbitrary element of $\cont U$. 
 \item (Consistency) For every $x\in X$, the tuple $(\pi_U(x))_{U\in \s}$ is $\kappa_0$-consistent (defined below). If $U\nest V$, then $d_W(\rho^U_W,\rho^V_W)\leq\kappa_0$ whenever $W\in\s$ satisfies either $V\propnest W$ or $V\trans W$ and $U\notorth W$. 
 \item (Bounded geodesic image) There exists $E>0$ such that for all $W\in\s$, all $V\in\s_W^{\circ}$, and all $x,y\in \X$ such that some geodesic from $\pi_W(x)$ to $\pi_W(y)$ stays $E$-far from $\rho^V_W$, we have $d_V(\pi_V(x),\pi_V(y))\leq E$. We will refer this property as BGI in this paper. 
\end{itemize}

\begin{defn}[$\kappa$-consistent tuple]
For a number $\kappa\geq 0$, let $\vec b=(b_U)_{U\in \s}\in\prod_{U\in \s}{2^{\C U}}$ be a tuple such that every set $b_U$ has diameter at most $\kappa$. We say that $\vec b$ is $\kappa$-consistent if 
\[
\min\big\{d_U(b_U,\rho^V_U),d_V(b_V,\rho^U_V)\big\}\leq\kappa \hspace{2mm}\text{ whenever } U\trans V, \text{ and}
\]\[
\min\big\{d_V(b_V,\rho^U_V),\diam_U(b_U\cup\rho^V_U(b_V))\big\}\leq\kappa \hspace{2mm}\text{ whenever } U\propnest V.
\]
\end{defn}

For convenience, we always take $E$ to be the greatest constant in all coarseness from the above list (see \cite[Remark 1.6]{BHS19} for discussion on these constants). For the rest of this subsection, let $(\X,\s)$ be a relative HHS.

\begin{notation}
Given $x,y\in \X$, we write $d_U(x,y)$ to mean $d_U(\pi_U(x),\pi_U(y))$. If $U\trans V$ or $U\propnest V$, then we write $d_V(x,\rho^U_V)$ to mean $d_V(\pi_V(x),\rho^U_V)$. 
\end{notation}

\begin{notation}
Given two functions $f,g:X\to \R$ and $A,B> 0$, we write $f\preceq_{(A,B)}g$ to mean $f(x)\leq Ag(x)+B$ for any $x\in X$. We write $f\asymp_{(A,B)} g$ to mean $\frac{1}{A}f(x)-B\le g(x) \le Af(x)+B$ for any $x\in X$. Sometimes we omit the constants to mean that the inequality holds for some constants.
\end{notation}

The powerful Masur--Minsky distance formula \cite{MM00} shows that the distance between points in a mapping class group is coarsely the sum of the distances between the projections of these points to the curve graphs of all subsurfaces. Like mapping class groups, relative HHSes also satisfy a Masur–Minsky-style distance formula. 

\begin{thm}[Distance formula]\cite[Theorem~6.10]{BHS19}\label{thm:distance_formula}
 There exists $s_0$ such that for all $s\geq s_0$ there exists a
 constant $C>0$ such that for all $x,y\in\X$,
 $$d_{\X}(x,y)\asymp_{(C,C)}\sum_{W\in\s}\ignore{d_{W}(x,y)}{s}.$$
\end{thm}

Closely related to the distance formula is the existence of \emph{hierarchy paths}. 

\begin{defn}[Hierarchy path]
A $(D,D)$-quasi-geodesic $\gamma\subset \X$ is a $D$-\emph{hierarchy path} if $\pi_U(\gamma)$ is an unparametrized $(D,D)$-quasi-geodesic for each $U\in\s$.
\end{defn}

\begin{thm}[Existence of hierarchy paths]\cite[Theorem~6.11]{BHS19}\label{thm:hierarchy_path}
 There exists $D_0$ such that any two points in $\X$ are joined by a $D_0$-hierarchy path.
\end{thm}

\begin{rem}\label{hierarchy}
    Let $\gamma$ be a $D$-hierarchy path connecting $x$ and $y$. By the construction of hierarchy paths in \cite{BHS19}, $\pi_U(\gamma)$ is contained in the $D$-neighborhood of a geodesic connecting $\pi_U(x)$ and $\pi_U(y)$. If $\C U$ is hyperbolic, this is easy to see from the Morse Lemma. In the general case, this deserves its own mention. 
\end{rem}

There is an important class of subspaces in relative HHSes. We will consider them in Section \ref{sec_appoximation}. 

\begin{defn}[Standard product region, standard nested region, standard orthogonal region]\label{std_product}
Fix $U\in\mathfrak S$ and $\kappa\geq\kappa_0$. Let $\F_U$ be the set of $\kappa$-consistent tuples in $\prod_{V\in\mathfrak S_U}2^{\C V}$. Let $\E_U$ be the set of $\kappa$-consistent tuples in $\prod_{V\in\mathfrak S_U^\orth-\{A\}}2^{\C V}$. Let $\pr_U=\F_U\times \E_U$. We can define a coarse map $\phi_U:\pr_U\to\X$ as follows. 

For each $(\vec a,\vec b)\in \F_U\times \E_U$, and each $V\in\mathfrak S$, define the coordinate $(\phi_U(\vec a,\vec b))_V$ as follows.  If $V\nest U$, then $(\phi_U(\vec a,\vec b))_V=a_V$.  If $V\orth U$, then $(\phi_U(\vec a,\vec b))_V=b_V$.  If $V\trans U$, then $(\phi_U(\vec a,\vec b))_V=\rho^U_V$. And if $U\propnest V$, then $(\phi_U(\vec a,\vec b))_V=\rho^U_V$. We can check that the tuple $\phi_U(\vec a,\vec b)$ is $\kappa$-consistent, and thus the realization theorem \cite[Theorem 3.1]{BHS19} supplies the map $\phi_U:\pr_U\to\X$ (see \cite[\S 5B]{BHS19} for more details). 

For convenience, we do not distinguish between $\pr_U$ and its image in $\X$. We call $\pr_U$ the \emph{standard product region}. By choosing any copy of $\F_U$ in the direct product, $\phi_U$ restricts to a coarse map $\phi^\nest:\F_U\to\X$. We also define $\phi^\orth:\E_U\to\X$ in the same way. We call $\F_U$ and $\E_U$ the \emph{standard nested region} and the \emph{standard orthogonal region}, respectively. 
\end{defn}

\begin{rem}
    By definition, $\F_U$, $\E_U$ and $\pr_U$ depend on the constant $\kappa$. In this paper, we fix some $\kappa\ge \kappa_0$ and do not mention it again.
\end{rem}

It is known that $(\F_U,\s_U)$, $(\E_U,\s_U^\orth)$ are both relatively hierarchically hyperbolic. By definition of $\F_U$, $\E_U$ and $\pr_U$, there are natural retractions from $\X$ to these subspaces. We call such a map a \emph{gate map}. Take $\F_U$ for example. We denote the gate map to $\F_U$ by $\g_{\F_U}$. For all $x\in\X$ and all $V\in\s$ such that $\C V$ is hyperbolic, $\pi_V(\g_{\F_U}(x))$ uniformly coarsely coincides with the closest point projection of $\pi_V(x)$ to the quasi-convex subset $\pi_V(\F_U)$. In fact, gate maps can be defined for all ``hierarchically quasi-convex'' subspaces, which form a larger class of subspaces of relative HHSes (see \cite[\S 5]{BHS19} for HHSes). 

For any (relative) HHS $(\X,\s)$, an automorphism is roughly speaking a bijection from $(\X,\s)$ to itself that preserves its (relative) HHS structure \cite[\S 1G]{BHS19}. The automorphisms of $(\X,\s)$ form a group $\aut(\s)$, which we call the \emph{automorphism group} of $(\X,\s)$. 

\begin{defn}[(Relatively) hierarchically hyperbolic groups]
A finitely generated group $G$ is \emph{(relatively) hierarchically hyperbolic}
if there exists a (relatively) hierarchically hyperbolic space $(\X,\s)$ and an action $G\to\aut(\s)$ such that the action $G\curvearrowright X$ is metrically proper and cobounded, and such that the induced action on $\s$ is cofinite. 
\end{defn}

Note that if $G$ is (relatively) hierarchically hyperbolic by virtue of its action on the (relatively) hierarchically hyperbolic space $(\X,\s)$, then $(G,\s)$ is a (relatively) hierarchically hyperbolic structure with respect to any word-metric on $G$. 

Let $\aut(\s;V)$ be the group of automorphisms $g\in\aut(\s)$ such that $g\cdot V=V$. Then there is a \emph{restriction homomorphism} $\theta_V:\aut(\s;V)\to\aut(\s_V)$ defined as follows. Given $g\in\aut(\s;V)$, let $\theta_V(g)$ act like $g$ on the substructure $\s_V$. For a group $G < \text{Aut}(\s)$, we write $\stab_G(V)$ to mean $G\cap\aut(\s;V)$ and write $G_V$ to mean the image of $\stab_G(V)$ under $\theta_V$. 

For many HHGs (for example, the case of mapping class groups), every $G_V$ acts acylindrically on $\C V$. However, not all HHGs have this property \cite{DHS20}.

\begin{defn}[colorability]\label{bbfcolor}
Let $(G,\s)$ be a relative HHG. Let $\s'\subset \s$ be a $G$-invariant subset. We say $\s'$ is \emph{colorable} if $\s'$ admits a decomposition $\s'=\bigsqcup_{i=1}^\chi\s'_i$ into finitely many $G$-invariant families $\s'_i$ such that any two domains in the same family are transverse. Such a decomposition is called a \emph{coloring} of $\s'$. We say a relative HHG $(G,\s)$ is \emph{colorable} if $\s$ is colorable. 
\end{defn}

The notion of colorability is formalized in \cite{DMS23,HP22}. There are many classes of (virtually) colorable HHGs, as listed in those papers. In particular, a coloring is constructed for (a finite-index subgroup of) a mapping class group in \cite[\S5]{BBF15}, from which the notion comes. However, one cannot expect that all HHGs are virtually colorable \cite{Hag23}. Nevertheless, \cite[Proposition~3.2]{HP22} provides a sufficient condition for an HHG to be virtually colorable. 

\begin{rem}
    In this paper, we only consider unbounded domains, i.e. domains with unbounded associated hyperbolic spaces. As an abuse of terminology, we say a relative HHG $(G,\s)$ is colorable if the collection of unbounded domains is colorable. 
\end{rem}

\section{Thick distance formula}
\label{sec_thickdistanceformula}

In this section, we will prove a \emph{thick distance formula} that is similar to \cite[Theorem 4.13]{BBF21}. This allows us to estimate the distance in a relative HHS by counting only ``thick'' segments of a hierarchy path instead of the whole hierarchy path. The reader should be aware that the definitions in this section are different from those in \cite{BBF21}. In particular, we do not have tight geodesics in a general HHG. 

Let $(\X,\s)$ be a relative HHS and fix $T>100E+10D_0$ (see Section \ref{def2} for constants associated with a relative HHS). As in \cite[\S2B]{BHS19}, we say a domain $U\in\s$ is $T$-\emph{relevant} for $x,y\in \X$ if $d_U(x,y)>T$. We write $\rel_T(x,y)$ for the set of $T$-relevant domains for $x,y$, and define $\rel_T(V;x,y):=\rel_T(x,y)\cap \s_V^{\circ}$. We write $\rel_T^{m}(V;x,y)$ for the set of $\nest$-maximal elements in $\rel_T(V;x,y)$. When $x$ and $y$ are fixed, we often omit them from the notation.

\begin{lem} \label{at_most_two}
Given $x,y\in\X$ and $U\in\rel_T(x,y)$, there exist at most two domains $V_1,V_2\in\rel_T(x,y)$ such that $U\in \rel_T^{m}(V_i;x,y)$ for $i=1,2$.
\end{lem}

\begin{proof}
Suppose there exist three such domains $V_1,V_2,V_3$. Since $U$ is maximal in each $\rel_T(V_i;x,y)$, we know that the $V_i$ are not $\nest$-comparable. Moreover, any two of them are not orthogonal since $U \propnest V_i$. Thus, $V_i$'s must be pairwise transverse. 

By \cite[Proposition 2.8]{BHS19}, any set of pairwise transverse elements in $\rel_T(x,y)$ has a total order $<$, obtained by setting $U<V$ whenever $d_U(y,\rho^V_U)\leq E$. We assume that $V_1<V_2<V_3$. 

On the one hand, $d_{V_2}(\rho^{V_1}_{V_2},\rho^{V_3}_{V_2})\geq d_{V_2}(x,y)-2E>T-2E$ by the triangle inequality. On the other hand, $d_{V_2}(\rho^{V_1}_{V_2},\rho^{U}_{V_2})\leq \kappa_0$ and $d_{V_2}(\rho^{V_3}_{V_2},\rho^{U}_{V_2})\leq \kappa_0$ by consistency, which gives $d_{V_2}(\rho^{V_1}_{V_2},\rho^{V_3}_{V_2})\leq 2\kappa_0<T-2E$. This gives a contradiction. 
\end{proof}

\begin{defn}[$T$-thickness]\label{thick}
Given $\s'\subset \s$, we say a pair of points $(x,y)\in \X\times\X$ is \emph{$T$-thick for $\s'$} if $\diam (\pi_U(x)\cup\pi_U(y)) \le T$ for all $U\in \s'$. We define $\PP_T(\s')$ to be the set of all $T$-thick pairs of points for $\s'$. If $\s'=\{U\}$, we also say $(x,y)$ is $T$-thick for $U$ and write $(x,y)\in \PP_T(U)$ .
\end{defn}

Note that $(x,y)\in \PP_T(U)$ if and only if $U\notin \rel_T(x,y)$. Also note that $(x,y)\in \PP_T(\s_V^{\circ})$ if and only if $\rel_T(V;x,y)=\emptyset$.

\begin{lem}\label{geodthick}
Let $D_0$ be the constant provided by Theorem \ref{thm:hierarchy_path}. For any $x,y\in\X$, let $\gamma$ be a $D_0$-hierarchy path between $x,y$. Given any $U\in \s$ and any $x',y'\in\gamma$, then
\[d_U(x',y')\le d_U(x,y)+2D_0.\]
In particular, if $(x,y)\in \PP_T(U)$, then $(x',y')\in \PP_{T+2D_0}(U)$.
\end{lem}

\begin{proof}
By Remark \ref{hierarchy}, $\pi_U(\gamma)$ lies in the $D_0$-neighborhood of a geodesic connecting $\pi_U(x)$ and $\pi_U(y)$. The conclusion then follows from the triangle inequality. 
\end{proof}

\begin{notation}
Throughout this paper, let $\hat T = T +2D_0$ and $\check T = T-2D_0$ for every constant $T>2D_0$.
\end{notation}

\begin{notation}
Given two points $x,y$ in a hyperbolic space, we write $[x,y]$ to mean a geodesic segment between $x,y$, which is coarsely unique. For an interval $I$ or a path $\gamma$, we write $I^-,I^+$ or $\gamma^-,\gamma^+$ to mean their endpoints. 
\end{notation}

\begin{defn}[$(T,R)$-thick distance]\label{trthick}
Fix sufficiently large constants $T,R$. Let $\gamma$ be a $D_0$-hierarchy path between $x$ and $y$. Let $\gamma_1,\dots,\gamma_n \subset \gamma$ be disjoint subpaths occurring in this order such that $(\gamma_i^-,\gamma_i^+)\in\PP_T(\s_V^{\circ})$ for each $i$. 

The {\it $(T,R)$-thick distance in $V$} is denoted by $d_V^{T,R}(x,y)$ and is defined to be the supremum of $\sum_{i=1}^{n} \ignore{d_V(\gamma_i^-,\gamma_i^+)}{R}$ over all such choices for $\gamma_i$'s, and for all $D_0$-hierarchy paths from $x$ to $y$. 
\end{defn}

It is always true that $d_V^{T,R}(x,y)\le \ignore{d_V(x,y)}{R}$. This becomes an equality if $V$ is $\nest$-minimal. For the opposite direction, we have the following estimate. 

\begin{lem}\label{thick_distance}
Fix constants $T,R>100E$. For any $x,y\in\X$ and $W\in\s$, we have 
\[\ignore{d_W(x,y)}{R}\leq d^{T,R}_W(x,y)+(6E+2R)|\rel_{\check T}^{m}(W;x,y)|.\]
\end{lem}

\begin{proof}
If $\rel_{\check T}^{m}(W)=\emptyset$, then $(x,y)\in \PP_T(\s_W^{\circ})$. Thus, both sides of the above inequality are equal.

Now assume that $\rel_{\check T}^{m}(W)\neq \emptyset$ and $d_W(x,y)\ge R$. Then $\C W$ is not $\nest$-minimal so it is hyperbolic. Let $\gamma:I\to \X$ be a $D_0$-hierarchy path realizing $d^{T,R}_W(x,y)$, where $I$ is an interval of $\R$. For any $V\propnest W$, we define 
\[s_V^-:= \inf\{s\in I\mid\exists U\nest V \text{such that} d_W(\gamma(s),\rho^U_W)\leq 2E\},\]
\[s_V^+:=\sup\{s\in I\mid\exists U\nest V \text{such that} d_W(\gamma(s),\rho^U_W)\leq 2E\}.\] 
For any $U\nest V\propnest W$, we know that $d_W(\rho^U_W,\rho^V_W)\leq \kappa_0$ by consistency. Thus, 
\[d_W(\gamma(s_V^-),\rho^V_W))\le 2E+\kappa_0\le 3E,\] 
\[d_W(\gamma(s_V^+),\rho^V_W))\le 2E+\kappa_0\le 3E.\] 
Therefore,
\[d_W(\gamma(s_V^-),\gamma(s_V^+))\leq d_W(\gamma(s_V^-),\rho^V_W))+d_W(\gamma(s_V^+),\rho^V_W))\le 6E.\]

Let $J_0,\dots,J_n$ be the collection of maximal intervals in $I-\bigcup_{V\in \rel_{\check T}^{m}(W)}(s_V^-,s_V^+)$. Note that $n\leq |\rel_{\check T}^{m}(W)|$. Now we are going to prove that $(\gamma(J_i^-),\gamma(J_i^+))\in\PP_T(\s_W^{\circ})$. 

On the one hand, $d_W(\rho_W^U,\gamma(J_i))\ge 2E$ for any $U\in \rel_{\check T}(W)$ by the definition of $J_i$. By the Morse Lemma, 
\[d_W(\rho_W^U,[\gamma(J_i^-),\gamma(J_i^+)])\ge d_W(\rho_W^U,\gamma(J_i))-E\ge E.\]
Therefore, $d_U(\gamma(J_i^-),\gamma(J_i^+))\le E <T$ by BGI. On the other hand, $(x,y)$ is $\check T$-thick for $\s_W^{\circ}-\rel_{\check T}(W)$ by definition. It follows from Lemma \ref{geodthick} that $(\gamma(J_i^-),\gamma(J_i^+))$ is $T$-thick for $\s_W^{\circ}-\rel_{\check T}(W)$. In sum, $(\gamma(J_i^-),\gamma(J_i^+))\in\PP_T(\s_W^{\circ})$. 

Finally, we estimate that
\begin{align*}
d_W(x,y) & \leq\sum_{i=0}^nd_W(\gamma(J_i^+),\gamma(J_i^-))+6E|\rel_{\check T}^{m}(W)|\\
& \leq d^{T,R}_W(x,y)+R(n+1)+6E|\rel_{\check T}^{m}(W)|\\
& \leq d^{T,R}_W(x,y)+(6E+2R)|\rel_{\check T}^{m}(W)|.
\end{align*}
\end{proof}

Let $S$ denote the unique maximal domain in $\s$. Recall that the level $\ell(S)$ of $S$ is equal to the complexity of $(\X,\s)$. 

\begin{thm}\label{thm:thick_distance}
Fix the constants $T,R$ with $\check T\ge R>100E$.
Let $x,y \in \X$.
Then, for each $n$
$$\sum_{\ell(W) \le n} \ignore{d_W(x,y)}{R} \le \sum_{\ell(W) = n} d^{T,R}_W(x, y) + 7\sum_{\ell(W)<n} \ignore{d_W(x,y)}{R}.$$
\end{thm}

Note that each sum has finitely many terms since there are only finitely many $W$ such that $d_W(x, y) \ge R$ for given $x,y$ by the distance formula (Theorem \ref{thm:distance_formula}). 

\begin{proof}
If $\ell(W) = n$ then by Lemma \ref{thick_distance},
\begin{align*}
\ignore{d_W(x,y)}{R} &\le d^{T,R}_W(x, y) + (6E+2R)|\rel^{m}_{\check T}(W)| \\
&\le d^{T,R}_W(x, y) + (6E+2R)\sum_{V \in \rel^{m}_{\check T}(W)} \frac{\ignore{d_V(x, y)}{\check T}}{\check T} \\
&\le d^{T,R}_W(x, y) + 3\sum_{V \in \rel^{m}_{\check T}(W)} \ignore{d_V(x, y)}{\check T} \\
&\le d^{T,R}_W(x, y) + 3\sum_{V \in \rel^{m}_{\check T}(W)} \ignore{d_V(x, y)}{R}.
\end{align*}
By Lemma \ref{at_most_two}, any $V$ appears in at most two $\rel^{m}_{\check T}(W)$. Therefore, if we sum up the left side over all $W$ with $\ell(W) = n$, we have
$$\sum_{\ell(W) =n} \ignore{d_W(x,y)}{R} \le \sum_{\ell(W)=n} d^{T,R}_W(x, y) + 6\sum_{\ell(W)<n} \ignore{d_W(x,y)}{R}.$$
Adding $\sum_{\ell(W)<n} \ignore{d_W(x,y)}{R}$ to both sides gives the desired inequality.
\end{proof}

\begin{cor}\label{cor.thick_distance}
Fix the constants $T,R$ with $\check T\ge R>100E$. Let $x,y \in \X$. 
Then
$$\frac{1}{D_0}\sum_{W\in\s} d_W^{T,R}(x,y)-D_0
 \le 
\sum_{W\in\s} \ignore{d_W(x,y)}{R} 
\le 
7^{\ell(S)-1} \sum_{W\in\s} d_W^{T,R}(x,y).$$
\end{cor}

\begin{proof}
The first inequality is trivial since 
$\frac{1}{D_0}d_W^{T,R}(x,y)-D_0 \le \ignore{d_W(x,y)}{R} $ for all $W$.
By inductively applying Theorem \ref{thm:thick_distance}, with base case $n=\ell(S)$, we have
$$\sum_{W\in\s} \ignore{d_W(x,y)}{R} \le 7^{\ell(S) - n}\left(\sum_{n \le \ell(W) \le \ell(S)} d^{T,R}_W(x,y) + 7\sum_{\ell(W)<n} \ignore{d_W(x,y)}{R}\right).$$
When $n=1$, the last term on the right is zero, and the result follows.
\end{proof}

Combining the distance formula (Theorem \ref{thm:distance_formula}) with Corollary \ref{cor.thick_distance}, we obtain our thick distance formula.

\begin{thm}[Thick distance formula]\label{thm.thick.dist}
There exists $R_0$ such that for all $T,R$ with $\check T\ge R>R_0$, there exists a constant $L>0$ such that for all $x,y \in \X$,
\[d_{\X}(x,y)\asymp_{(L,L)} \sum_{W\in\s} d_W^{T,R}(x,y).\]
\end{thm}

\section{Estimation of thick distance via quasi-axes}
\label{sec_appoximation}

Our proof in this section is inspired by \cite{NY23a}. The main technique in the proof of \cite[Lemma 5.5]{NY23a} that is different from \cite{BBF21} is the use of the Extension Lemma \cite[Lemma 2.14]{Yan19}. Lemma \ref{ExtensionLem} below is a trimmed version for acylindrical actions on hyperbolic spaces. 
In the original statement of Lemma \ref{ExtensionLem} in \cite{NY23a}, the group action is required to be cobounded, but it is easy to see from the proof that this condition can be removed. 

\begin{lem}[Extension Lemma]\label{ExtensionLem}\cite[Lemma~4.13]{NY23a}
Let $H$ be a group acting non-elementarily and acylindrically on a $\delta$-hyperbolic space $Y$. Fix a base point $o\in Y$. There exists a set $F\subset H$ of three loxodromic elements and constants $\lambda\ge 1, c\ge 0$  with the following property. 

For any $h\in H$ there exists $f\in F$ such that $hf$ is a loxodromic element and the bi-infinite path 
$$\gamma_h=\bigcup_{i\in \mathbb Z}(hf)^{i}([o, ho][ho, hfo])$$
is a $(\lambda, c)$-quasi-geodesic.
\end{lem} 

Let $(\X,\s)$ be a relative HHS with a coarse constant $E$, and let $G$ be a relative HHG by virtue of its action on $\X$. 
Corresponding to Definition \ref{type}, we say that a domain $U\in \s$ has
    \begin{enumerate}
     \item \emph{hyperbolicity} if $\C U$ is hyperbolic.
     \item \emph{acylindrical image} if $G_U$ acts on $\C U$ acylindrically.
     \item \emph{cobounded nested region} if $G_U$ acts on $\F_U$ coboundedly.
     \item \emph{separable quasi-axes} if for any element $g\in\stab_G(U)$ that acts loxodromically on $\C U$, the elementary closure $EC(g)$ is \emph{separable} in $G$, i.e. $EC(g)$ equals the intersection of all finite-index subgroups of $G$ that contain $EC(g)$. 
    \end{enumerate}

Till the end of this section, let $V\in \s$ be an unbounded domain that has hyperbolicity, cobounded nested region and acylindrical image. 
The next lemma could be compared with \cite[Theorem 4.19]{BBF21} for mapping class groups. 

\begin{lem}[Extension of thick segments]\label{Ext.Thick}
    There exist constants $\lambda\ge 1,c\ge 0,B\ge 0$ such that the following holds. For any $T,R>0$, there exists a $G_V$-finite collection $\A_V=\A_V^{T,R}$ of $(\lambda,c)$-quasi-axes in $\C V$ such that for any pair of points $(x,y)\in\PP_{\max\{\hat T, R\}}(\s_V)$, there exists $\gamma\in \A_V$ such that $[\pi_V(x),\pi_V(y)]\subset \neb_B(\gamma)$.
\end{lem}

\begin{proof}
    Fix a base point $o\in \F_V$ and project it to a base point in $\C V$. If the action $G_V\curvearrowright \C V$ is elementary, $\C V$ itself is a quasi-axis of a loxodromic element, and thus satisfies the requirement. Now assume that the action $G_V\curvearrowright \C V$ is non-elementary. Lemma \ref{ExtensionLem} provides a finite set $F\subset G_V$ and constants $\lambda\ge 1, c\ge 0$. Recall that $(\F_V,\s_V)$ is a relative HHS. $\F_V$ is proper because $\X$ is proper. 
    
    Since $G_V$ acts coboundedly on $\F_V$, there exists $\epsilon>0$ such that $\F_V$ is covered by the $G_V$-translates of any $\epsilon$-ball. Let $T'=\max\{\hat T, R\}$. 
    By Theorem \ref{thm:distance_formula}, there exists $r>0$, depending only on $E,T',\epsilon$, such that the distance between any pair of points in $\F_V$ that is $(T'+\epsilon)$-thick for $\s_V$ is bounded above by $r$. Fix any base point $o\in \F_V$. 
    Since $\F_V$ is proper, there exists a finite subset $S\subset G_V$ such that $\neb_{r+\epsilon}(o)$ is covered by $\bigcup_{s\in S}s\cdot \neb_{\epsilon}(o)$. 

    Lemma \ref{ExtensionLem} tells us that for each $s\in S$, there exists $f\in F$ such that $sf$ is a loxodromic element acting on $\C V$. Let $\A_V=\A_V^{T,R}$ be the collection of $G_V$-translates of the $(\lambda,c)$-quasi-axes provided by Lemma \ref{ExtensionLem} for all loxodromic elements of the form $sf$. 

    Now we verify that $\A_V$ meets our requirements. 
    Let $(x,y)\in\PP_{T'}(\s_V)$. We can choose $g\in G_V$ such that $d_{\F_V}(x,go)<\epsilon$. Then $d_{\F_V}(o, g^{-1}y)\le d_{\F_V}(x,y)+d_{\F_V}(x,go)<r+\epsilon$. 
    By our choice of $S$, there exists $s\in S$ such that $d_{\F_V}(g^{-1}y,so)<\epsilon$. Thus, we have
    \begin{align*}
        d_V(x,go)&<E\epsilon+E,\\
        d_V(y,gso)&<E\epsilon+E,
    \end{align*}
    because $\pi_V$ is $E$-coarsely Lipschitz. 
    Since $\C V$ is $E$-hyperbolic, we can find $B>0$ by fellow-traveller property such that $[\pi_V(x),\pi_V(y)]\subset \neb_{B}([g\cdot \pi_V(o),gs\cdot \pi_V(o)])$. By construction, $[g\cdot \pi_V(o),gs\cdot \pi_V(o)]$ is contained in some $\gamma\in \A_V$ so we are done.
\end{proof}

\begin{notation}
    Assume that $\gamma$ is a quasi-geodesic in a hyperbolic space $Y$. We write $\pi_{\gamma}:Y\to \gamma$ to mean the closest point projection. For $x,y\in Y$, we write $d_{\gamma}(x,y)$ to mean $\diam (\pi_{\gamma}(x)\cup\pi_{\gamma}(y))$. 
\end{notation}

\begin{notation}
    Assume that $\gamma$ is a quasi-geodesic in $\C V$. We write $\pi^\X_{\gamma}$ to mean $\pi_{\gamma} \circ \pi_V$. For $x,y\in \X$, we write $d^{\X}_{\gamma}(x,y)$ to mean $\diam (\pi^{\X}_{\gamma}(x)\cup \pi^{\X}_{\gamma}(y))$.
\end{notation}

The following lemma is a well-known corollary of the Morse Lemma for quasi-geodesics in $\delta$-hyperbolic spaces so we omit the proof. 

\begin{lem}\label{ProjBd.hyp}
    Let $\gamma$ and $\alpha$ be two $(\lambda, c)$-quasi-geodesics in a $\delta$-hyperbolic space. Then for any $B>0$, there exists a constant $C=C(\lambda, c, B, \delta)>0$ such that 
    \begin{align*}
    d_\gamma(\alpha^-,\alpha^+)\ge \diam(\alpha\cap \neb_B(\gamma))- C.
    \end{align*}
\end{lem}

The main result of this section is the following estimate that generalizes \cite[Proposition 4.18]{BBF21}. 

\begin{prop}\label{thickandline}
For any $K>0$, there exists $R>0$ such that the following holds. 
Given any $T>0$, let $\A_V=\A_V^{T,R}$ be the collection of $(\lambda,c)$-quasi-axes provided by Lemma \ref{Ext.Thick}. Then for any two points $x, y\in \X$, 
\begin{align*}
\label{ConeoffYvDistFormulaEQ}
d_{V}^{T,R}(x, y) \le 2(D_0+1)\sum_{\gamma\in \A_V} \ignore{d^{\X}_\gamma(x, y)}{K}
\end{align*}
where $D_0$ is the constant provided by Theorem \ref{thm:hierarchy_path}. 
\end{prop}

\begin{proof}
Let $C=C(\lambda+D_0, c+D_0, B, E)$ be the constant provided by Lemma \ref{ProjBd.hyp}. 
Let $R>2D_0(C+1)+K$. We will show that projections to quasi-axes $\A_V=\A_V^{T,R}$ bound the $(T,R)$-thick distance in $V$ from above. 

For any two points $x,y\in \X$, let $\beta$ be a $(D_0,D_0)$-hierarchy path connecting $x$ and $y$ realizing $d_V^{T,R}(x,y)$. Let $\{\alpha_1,\dots,\alpha_n\}$ be the collection of disjoint subpaths of $\beta$ with $d_V(\alpha_i^-,\alpha_i^+)\ge R$ and $(\alpha_i^-,\alpha_i^+)\in\PP_T(\s_V^{\circ})$ such that
\[d^{T,R}_V(x,y) = \sum_{i=1}^n d_V(\alpha_i^-,\alpha_i^+).\]

By the definition of gate maps, $\pi_V(z)$ is coarsely $\pi_V(\g_{\F_V}(z))$ for any $z\in \X$. Thus, the difference between $d_V(\alpha_i^-,\alpha_i^+)$ and $d_V(\g_{\F_V}(\alpha_i^-),\g_{\F_V}(\alpha_i^+))$ is uniformly bounded. This enables us to replace $\alpha_i$ with $\g_{\F_V}(\alpha_i)$ from now on. We divide each $\alpha_i$ into several consecutive subpaths $\{\tilde\alpha_{i,j}~|~1\le j\le m_i\}$ with $d_V(\tilde\alpha_{i,j}^-,\tilde\alpha_{i,j}^+)=R$ for $j=1,\dots,m_i-1$ and $d_V(\tilde\alpha_{i,m_i}^-,\tilde\alpha_{i,m_i}^+)\le R$. By Lemma \ref{geodthick}, we already know that $(\tilde\alpha_{i,j}^-,\tilde\alpha_{i,j}^+)\in\PP_{\hat T}(\s_V^{\circ})$ for every pair $(i,j)$. 
Thus, \[d_U(\tilde\alpha_{i,j}^-,\tilde\alpha_{i,j}^+)\le \max\{\hat T, R\}\] 
for all $U\in \s_V$. By Lemma \ref{Ext.Thick}, there exists $\gamma_{i,j}\in \A_V$ such that $\pi_U(\tilde\alpha_{i,j})\subset \neb_{B}(\gamma_{i,j})$ (with an increased $B$ by a uniform constant), which yields 
\[\diam(\pi_V(\alpha_i)\cap \neb_{B}(\gamma_{i,j}))\ge R.\]

Let $\A_V'$ be the collection of all distinct $\gamma_{i,j}$. We see that
\[\pi_V(\alpha_i \backslash \tilde\alpha_{i,m_i}) \subset \bigcup_{\gamma\in\A_V'} \pi_V(\alpha_i)\cap \neb_{B}(\gamma).\]
Thus, we have 
\begin{align*}
    d_V(\alpha_i^-,\alpha_i^+)&\le \sum_{\gamma\in\A_V'}\diam(\pi_V(\alpha_i)\cap \neb_{B}(\gamma))+R\\
    &\le 2\sum_{\gamma\in\A_V'}\diam(\pi_V(\alpha_i)\cap \neb_{B}(\gamma)).
\end{align*}
Summing up from $i=1$ to $n$ yields that
\begin{align}\label{thick.estimate.by.diam}
d^{T,R}_V(x,y)\le 2\sum_{\gamma\in\A_V'}(D_0\diam(\pi_V(\beta)\cap \neb_{B}(\gamma))+D_0).
\end{align}

Note that $R>2D_0(C+1)+K$. Thus, Lemma \ref{ProjBd.hyp} tells us that $d^{\X}_{\gamma}(x,y)\ge \diam(\pi_V(\beta)\cap\neb_{B}(\gamma))-C\ge R-C >D_0(C+1)+K$ for each $\gamma\in\A_V'$. We now estimate by Lemma \ref{ProjBd.hyp} that
\begin{align*}
D_0\diam(\pi_V(\beta)\cap \neb_{B}(\gamma))+D_0 &\le D_0(d^{\X}_{\gamma}(x,y)+C)+D_0\\
&<(D_0+1)d^{\X}_{\gamma}(x,y).
\end{align*}

Combining this with Equation (\ref{thick.estimate.by.diam}), we obtain that
\begin{align*}
d^{T,R}_V(x,y)&\le 2(D_0+1)\sum_{\gamma\in \A_V'}d^{\X}_{\gamma}(x,y)\\
&= 2(D_0+1)\sum_{\gamma\in \A_V'} \ignore{d^{\X}_\gamma(x, y)}{K}\\
&\le 2(D_0+1)\sum_{\gamma\in \A_V} \ignore{d^{\X}_\gamma(x, y)}{K}.
\end{align*}

\end{proof}

\section{Construction of quasi-trees}
\label{sec_quasitrees}

This section is devoted to the proof of Theorem \ref{maintheorem}. Let $(G,\s)$ be a relative HHG that is virtually colorable and assume that every domain in $\s$ is of type I or type II. The index set $\s$ admits a $G$-invariant decomposition $\s=\s^{I}\sqcup \s^{II}$, where $\s^I$ (respectively, $\s^{II}$) only contains domains of type I (respectively, type II). Note that type I and type II are not mutually exclusive, but for those domains of both types, we can simply put them in $\s^{II}$. 

Before starting the proof, we summarize the dependencies of some important constants that will be used in the proof as follows. 
    \[\xymatrix@C+2.5pc{(E,D_0,A) \ar[r]^{\qquad \text{Corollary \ref{fin_index}}} & \theta \ar[r]^{\text{Theorem \ref{proj.axioms}}} & \xi \ar[r]^{\text{Lemma \ref{thickandqtree}}} & K \ar[r]^{\text{Proposition \ref{thickandline}}} & R \ar[r]^{\text{Theorem \ref{thm.thick.dist}}} & T} \]
Here $A$ stands for the acylindrical constants. We draw an arrow from a constant $M$ to $N$ if $N$ depends on $M$. Remember that the dependency graph shown above is incomplete, but we hope it is helpful to the reader. 

\subsection{Quasi-trees from domains of type I}
\label{sec_quasitrees_a}

In this subsection, we are going to prove the following proposition.

\begin{prop}\label{typei}
    There exists a finite-index subgroup $H<G$ satisfying the following.    
    For any sufficiently large constant $R$ and any $T>0$, there exist quasi-trees $\T_1,\dots,\T_n$ such that $H$ acts on $\prod_{j=1}^n \T_j$ diagonally and for any choice of base points $o_j\in \T_j$ we have
    \begin{align*}
        \sum_{V\in\s^I}d_{V}^{T,R}(1, h) \preceq \sum_{j=1}^nd_{\T_j}(o_j, ho_j)
    \end{align*}
    for any $h\in H$.
\end{prop}

Before the proof, we recall the following proposition.

\begin{prop}\cite[Proposition~3.4]{BBF21}
\label{prop.bddproj}
    Let a group $H$ act on a $\delta$-hyperbolic space $Y$. Assume that the image of $H$ in $\text{Isom}(Y)$ is acylindrical. Consider a loxodromic element $g\in H$ and the collection $\mathbb A$ of all $H$-translates of a fixed $(\lambda,c)$-quasi-axis of $g$. Then there exists a constant $\theta>0$ depending only on $\lambda,c,\delta$ and the acylindrical constants such that for any $\gamma\in \mathbb A$, the set $$\{h\in H\mid \diam(\pi_\gamma(h\gamma))\ge  \theta\}$$ is a finite union of double $EC(g)$-cosets. 
\end{prop}

\begin{cor}
\label{fin_index}
    Let $U\in \s^I$. Consider a $(\lambda,c)$-quasi-axis $\gamma\subset \C U$ for some loxodromic element of the acylindrical action of $G_U$. Then there exists $\theta>0$, only depending on $\lambda,c,E$ and the acylindrical constants, and a finite-index subgroup $G_{\gamma}<G$ such that every translate of $\gamma$ by an element of $G_{\gamma}\cap \stab_G(U)$ either has finite Hausdorff distance with $\gamma$ or has $\theta$-bounded projection to $\gamma$. 
\end{cor}

\begin{proof}
    This is clear by Proposition \ref{prop.bddproj} and separability of quasi-axes. 
\end{proof}

By definition of relative HHGs, $\s^I$ consists of finitely many $G$-orbits, so acylindrical constants for $U\in \s^I$ can be chosen uniformly. Thus, Lemma \ref{Ext.Thick} provides uniform constants $\lambda\ge 1,c\ge 0$ for every $U\in \s^I$. This further gives a uniform constant $\theta>0$ by Corollary \ref{fin_index}. 

Let $\mathcal U$ be a $G$-representative set of $\s^I$ such that $1\in \pr_U$ for any $U\in \mathcal U$. Let $U\in \mathcal U$. Let $T>0$ and let $K>0$ be a sufficiently large constant that will be decided by Lemma \ref{thickandqtree}. Lemma \ref{Ext.Thick} provides a $G_U$-finite collection $\A_U=\A_U^{T,R}$ of $(\lambda,c)$-quasi-axes, where $R$ is provided by Proposition \ref{thickandline}. 
By Corollary \ref{fin_index}, we can find a finite-index subgroup $H_U<G$ such that for any $\gamma\in \A_U$ and $h \in H_U\cap \stab_G(U)$, either $d_{Haus}(h\gamma,\gamma) < \infty$ or $\diam \pi_{\gamma}(h\gamma) < \theta$. For any $g\in G$, we define $\A_{gU}:=\{g\gamma~|~\gamma\in \A_U\}$. 

Let $\A := \bigsqcup_{U\in \s} \A_U$ and $H := \bigcap_{U\in \mathcal U} H_U$.
Since $\mathcal U$ is finite, $H$ is of finite index in $G$. By adding finitely many domains to $\mathcal U$ so that there is one representative for each $H$-orbit on $\s^I$, we obtain an $H$-representative set $\tilde{\mathcal U}$ of $\s^I$. We still assume that $1\in \pr_U$ for any $U\in \tilde{\mathcal U}$. 
Let $\{\gamma_1,\dots,\gamma_n\}$ be an $H$-representative set of $\A$. We assume that every representative $\gamma_j$ is contained in $\C U$ for some $U\in \tilde{\mathcal U}$. Let $\A_j\subset \A$ be the $H$-orbit of $\gamma_j$.

Without loss of generality, we assume that $H$ is colorable instead of virtually colorable. Thus, the $H$-orbit of any domain is pairwise transverse. By \cite[Lemma 3.4]{HP22}, every $H$-orbit of $\s^I$ is an $H$-projection system with constant $s_0+4E$, where $s_0$ is the constant provided by Theorem \ref{thm:distance_formula}. Thus, every $\A_j$ is an $H$-projection system with a uniform projection constant $\xi=\xi(\theta,s_0,\lambda,c,E)$ by Theorem \ref{proj.axioms}. The projections defined there will be denoted by $\Pi_{\gamma}$. 

Using Theorem \ref{strong axioms}, we obtain modified projections $\Pi'_{\gamma}$ such that $(\A_j,\{\Pi'_{\gamma}\})$ satisfies the strong projection axioms with constant $\xi'=\xi'(\xi)$ and that $\Pi_{\gamma}(\alpha)$ and $\Pi'_{\gamma}(\alpha)$ are apart from each other within a uniform Hausdorff distance $\epsilon=\epsilon(\xi)$. For any $K'\ge 4\xi'$, $\C_{K'}\A_j$ is a quasi-tree by Theorem \ref{bbf}. The following lemma is an estimate via the orbit map between the projections $\pi^H_{\gamma}=\pi_{\gamma} \circ \pi_V$ in the relative HHG structure and the projections $\Pi'_{\gamma}$ in the quasi-tree $\C_{K'}\A_j$. Recall from Section \ref{sec_appoximation} that $d^H_{\gamma}(g,h)$ is defined as $\diam (\pi^H_{\gamma}(g)\cup \pi^H_{\gamma}(h))$ for any $g,h \in H$.

\begin{lem}\label{thickandqtree}
    Fix a base point $o_j\in\gamma_j$ for each $j=1,\dots,n$. There exists a sufficiently large constant $K'=K'(\xi, \lambda, c, E)$ and a constant $\Delta>0$ such that if $K\ge 2K'$, then 
    \[\sum_{\gamma\in\A_j}\ignore{d^H_{\gamma}(1,h)}{K}\le 8d_{\C_{K'}\A_j}(o_j,h o_j)+\Delta\]
    for any $h\in H$ and any $j=1,\dots,n$.
\end{lem}

For the proof of Lemma \ref{thickandqtree}, we need the following lemma. 

\begin{lem}\label{threshold_estimate}
    For any constants $A,B\ge 0$ and constants $L,M> 0$, 
    \[\frac{\ignore{A+B}{L+M}}{L+M}\le \frac{\ignore{A}{L}}{L}+\frac{\ignore{B}{M}}{M}.\]
\end{lem}

\begin{proof}
    Assume that $A+B\ge L+M$. First, if $A< L$ then $B> M$. Thus, $\ignore{B}{M} = B\ge \frac{M}{L+M}(A+B)$. Next, if $B< M$ then $A> L$ and the same argument holds. Finally, if $A\ge L$ and $B\ge M$, then all thresholds are reached and the inequality is obviously true. 
\end{proof}

\begin{proof}[Proof of Lemma \ref{thickandqtree}]
    For simplicity, we use $|p-q|$ to mean the distance between two points $p,q$ in the same space. Assume that $\gamma_j\subset \C U$ for some $U\in \tilde{\mathcal U}$. For any $g\in H-\stab_H(U)$, since $1\in \pr_U$ and $U\trans gU$, we have $|\pi_{gU}(1)-\rho^U_{gU}|\le E$ by the definition of standard product regions. 
    
    Assume $\gamma=g\gamma_j$. By hyperbolicity, there exists a constant $F=F(E,\lambda,c)$ such that if $\gamma\subset \C gU\ne \C U$ then $|\pi^H_{\gamma}(1)-\Pi'_{\gamma}(o_j)|\le |\pi_{gU}(1)-\rho^U_{gU}|+F+\epsilon \le E+F+\epsilon$. Let $M>\xi'+E+F+\epsilon$ and define $\delta_{\gamma}(h)=|\pi^H_{\gamma}(h)-\Pi'_{\gamma}(ho_j)|$. We see that if $\delta_{\gamma}(h)=\delta_{h^{-1}\gamma}(1)\ge M$ then $\gamma\subset \C hU$. Thus, for a fixed $h\in H$, there are only finitely many $\gamma\in \A_j$ such that $\delta_{\gamma}(h)\ge M$ by projection axiom (P2). 
    
    Let $K'>2M+4\xi'$. Define $D_{\gamma}(1,h)=|\Pi'_{\gamma}(o_j)-\Pi'_{\gamma}(ho_j)|$. By the triangle inequality and Lemma \ref{threshold_estimate}, we obtain that
    
    \[\ignore{d^H_{\gamma}(1,h)}{K'+2M}\le \frac{K'+2M}{K'}\ignore{D_{\gamma}(1,h)}{K'}+\frac{K'+2M}{2M}\ignore{\delta_{\gamma}(1)+\delta_{\gamma}(h)}{2M}.\]
    Therefore, 
    \[\ignore{d^H_{\gamma}(1,h)}{K}\le 2\ignore{D_{\gamma}(1,h)}{K'}+\frac{K'+2M}{M}\ignore{\delta_{\gamma}(1)}{M}.\]

    Summing over $\gamma\in\A_j$, we obtain that
    \begin{align*}
    \sum_{\gamma\in\A_j}\ignore{d^H_{\gamma}(1,h)}{K} & \le 2\sum_{\gamma\in\A_j}\ignore{D_{\gamma}(1,h)}{K'}+\frac{K'+2M}{M}\sum_{\gamma\in\A_j}\ignore{\delta_{\gamma}(1)}{M}.
    \end{align*}

    By the discussion above, $\Delta_j= \frac{K'+2M}{M}\sum_{\gamma\in\A_j}\ignore{\delta_{\gamma}(1)}{M}$ is a finite constant. Let $\Delta=\max_{1\le j\le n} \Delta_j$. We conclude using Theorem \ref{distfor2}.

\end{proof}

Now we can prove Proposition \ref{typei}.

\begin{proof}[Proof of Proposition \ref{typei}]
    Let $K$ and $K'$ be the constants provided by Lemma \ref{thickandqtree}. Let $T>0$. From the discussion before Lemma \ref{thickandqtree}, we know that the collection $\A$ of quasi-axes provided by Proposition \ref{thickandline} forms $n$ quasi-trees $\C_{K'}\A_j$, $j=1,\dots,n$. Moreover, Proposition \ref{thickandline} tells us that there exists $R>0$ such that  
    \[\sum_{V\in\s^I}d_{V}^{T,R}(1, h) \preceq \sum_{\gamma\in \A} \ignore{d^{H}_\gamma(1, h)}{K}.\]    
    Finally, we conclude by Lemma \ref{thickandqtree}.
\end{proof}

\subsection{Quasi-trees from domains of type II}
\label{sec_quasitrees_b}
In this subsection, we prove an analogue of Proposition \ref{typei} for domains of type II. 

Let $H<G$ be the subgroup provided by Proposition \ref{typei}. 
Let $\mathcal V$ be an $H$-representative set of $\s^{II}$ such that $1\in \pr_V$ for any $V\in \mathcal V$. For $V\in\mathcal V$ we write $[V]$ to mean its $H$-orbit. By colorability and \cite[Lemma 3.4]{HP22}, every $[V]$ is an $H$-projection system with constant $s_0+4E$. 

Fix any $V\in\mathcal V$. By property QT$_0$ of the action $\stab_G(V)\curvearrowright \C V$, there exist quasi-trees $T_V^i$ along with $\stab_H(V)$-equivariant maps $\iota^i_V:\C V\to T_V^i$ for $i=1,\dots,n_V$ such that 
\[\prod_{i=1}^{n_V}\iota_V^i:\C V\to \prod_{i=1}^{n_V}T_V^i\]
is a $(\lambda',c')$-quasi-isometric embedding. In particular, $\iota^i_V$ is $(\lambda',c')$-coarsely Lipschitz. Fix $i\in\{1,\dots,n_V\}$. It is conventional to extend the map $\iota^i_V$ on $[V]$ in an $H$-equivariant way. This means that we can construct a collection of quasi-trees $\bT_{[V]}^i = \{T_U^i~|~U\in[V]\}$ with an $H$-action and a collection of coarsely Lipschitz maps $\{\iota_U^i:\C U\to T_U^i~|~U\in[V]\}$ such that the following diagram commutes for any $h\in H$ and $U\in [V]$. 

\[\xymatrix{
\C U \ar[r]^h \ar[d]_{\iota_U^i} & \C hU \ar[d]^{\iota_{hU}^i} \\
T_U^i \ar[r]^h & T_{hU}^i
}\]

Define $\Pi_{T_{hU}^i}(T_U^i):=\iota_{hU}^i(\rho^U_{hU})$ for any $U\in [V]$ and $hU\ne U$. Clearly, these projections are $H$-equivariant and the projection axioms pass to the collection $(\bT_{[V]}^i,\{\Pi_{T_U^i}\})$ under coarsely Lipschitz maps $\{\iota_U^i\}$. 
We modify the projections within an error $\epsilon$ such that $(\bT_{[V]}^i,\{\Pi'_{T_U^i}\})$ satisfies the strong projection axioms with constant $\zeta=\zeta(s_0,\lambda',c',E)$. For any $K''\ge 4\zeta$, $\C_{K''}\bT_{[V]}^i$ is a quasi-tree by Theorem \ref{bbf}. Define $d_{T^i_U}(1,h):=|\iota^i_U(\pi_U(1))-\iota^i_U(\pi_U(h))|$ for any $U\in[V]$. 
For any $V\in \mathcal V$ and $i\in\{1,\dots,n_V\}$, fix a base point $o_V^i\in T_V^i$. 
The following proposition is an analogue of Proposition \ref{typei}. 

\begin{prop}\label{typeii}
    There exists a constant $K''=K''(\zeta,\lambda',c',E)$ such that if $R\ge 2K''$ then
    \begin{align*}
        \sum_{U\in\s^{II}}\ignore{d_{U}(1, h)}{R} \preceq \sum_{V\in\mathcal V}\sum_{i=1}^{n_V}d_{\C_{K''}\bT_{[V]}^i}(o_V^i,ho_V^i).
    \end{align*} 
    for any $h\in H$. 
\end{prop}

\begin{proof}
    Fix $V\in \mathcal V$ and $i\in\{1,\dots,n_V\}$. 
    For any $U\in [V]-\{V\}$, we have $|\iota^i_U(\pi_U(1))- \Pi_{T^i_U}'(o_V^i)| \le \lambda'|\pi_U(1)-\rho^V_U|+c'+\epsilon\le \lambda' E+c'+\epsilon$. Let $K''> 4\zeta+2(\lambda' E+c'+\epsilon)$. Similarly to the proof of Lemma \ref{thickandqtree}, we estimate that
    \begin{align*}
        \sum_{U\in[V]}\ignore{d_{T^i_U}(1,h)}{R} &\preceq \ignore{|\Pi_{T^i_U}'(o_i)-\Pi_{T^i_U}'(ho_i)|}{K''}\\
        &\preceq d_{\C_{K''}\bT_{[V]}^i}(o_V^i,ho_V^i).
    \end{align*}
    Here the first inequality follows from the triangle inequality and Lemma \ref{threshold_estimate}, and the second holds by Theorem \ref{distfor2}. Since the map $\prod_{i=1}^{n_U}\iota_U^i:\C U\to \prod_{i=1}^{n_U}T_U^i$ is a quasi-isometric embedding for any $U\in \s^{II}$, we conclude by summing the inequality over $1\le i\le n_V$ for all $V\in\mathcal V$. 
\end{proof}

\subsection{Proof of Theorem \ref{maintheorem}}
    \begin{proof}
    Let $R>0$ be large enough to satisfy Proposition \ref{typei}, Proposition \ref{typeii} and Theorem \ref{thm.thick.dist}. Let $T\ge R+2D_0$.
    By Proposition \ref{typei} and Proposition \ref{typeii}, there exists quasi-trees $\T_1,\dots,\T_m$ such that $H$ acts on $\prod_{k=1}^m \T_k$ diagonally and for any choice of base points $o_k\in \T_k$ and any $h\in H$,
    \begin{align*}
        \sum_{V\in\s^I}d_{V}^{T,R}(1, h)+\sum_{U\in\s^{II}}\ignore{d_{U}(1, h)}{R} \preceq \sum_{k=1}^md_{\T_k}(o_k, ho_k).
    \end{align*}

    By definition of thick distance, $d_{U}^{T,R}(1, h)\le \ignore{d_{U}(1, h)}{R}$. Thus, by Theorem \ref{thm.thick.dist}, 
    \[\sum_{V\in\s}d_{V}^{T,R}(1, h) \preceq \sum_{k=1}^md_{\T_k}(o_k, ho_k)\]
    for any $h\in H$.
    
    On the other hand, the orbit map from $H$ to $\prod_{k=1}^m \T_k$ is coarsely Lipschitz since $H$ is finitely generated. Therefore, $H$ embeds quasi-isometrically into $\prod_{k=1}^m \T_k$, which means that $H$ has property QT$_0$. Finally, we know that $G$ has property QT since property QT passes to any finite-index supergroup. 

    For the ``moreover'' part, first note that $G$ is coarse median for the same reason as \cite[Theorem 7.3]{BHS19}. The rest of the proof is just a combination of \cite{HP22,Pet21}. The proof of \cite[Proposition 3.9]{HP22} can be naturally generalized to quasi-trees from domains of type II. The proof in \cite[\S 3]{Pet21} for mapping class groups can be repeated verbatim to deal with quasi-trees from domains of type I. 
    \end{proof}

    For relative HHGs with only type II domains, we obtain the following stronger theorem.

    \begin{thm}\label{typeIIthm}
        Let $(G,\s)$ be a relative HHG that is colorable. If every $U\in \s$ is of type II, then $G$ has property QT$_0$. 
    \end{thm}

    \begin{proof}
        The proof is a simplified version of that of Theorem \ref{maintheorem}. Since $G$ is colorable and every domain is of type II, the finite-index subgroup $H$ in the above proof can be replaced with $G$ itself. This gives QT$_0$ rather than just QT.
    \end{proof}

\section{A criterion for having separable quasi-axes}
\label{neatker}

In this section, we provide a criterion for a relative HHG to have separable quasi-axes that is easy to use in application. 
For an acylindrical action on a hyperbolic space, we have seen in Section \ref{back1} that the elementary closure of any loxodromic element is a \emph{virtual centralizer}. Inspired by the discussion in \cite[\S 4.3]{BBF21}, the following lemma generalizes this fact. In general, we cannot expect the elementary closure to be a centralizer as in \cite[\S 4.3]{BBF21}, because it may contain a flip. 

\begin{lem}
\label{centralizer}
Let $G$ be a group acting on a $\delta$-hyperbolic space $X$ with an acylindrical image. Let $K$ be the kernel of this action. Assume that there is a subgroup $H<G$ such that $[H,K]=1$ and $H$ is mapped to a finite-index subgroup of $G/K$. Then for any loxodromic element $g\in G$, the elementary closure $EC_G(g)$ is a virtual centralizer in $G$ of some loxodromic element in $H$. 
\end{lem}

\begin{proof}
    Let $\bar G=G/K$, and let $\theta:G\to \bar G$ be the quotient map. 
    For any $g\in G$, denote the image $\theta(g)\in \bar G$ by $\bar g$. 
    
    Let $g\in G$ be a loxodromic element. Since $\theta(H)$ is a finite-index subgroup of $\bar G$, we can choose $h\in H$ and $n>0$ such that $\bar h=\bar g^n$. Since $\bar G$ acts acylindrically, $EC_{\bar G}(\bar g)=EC_{\bar G}(\bar h)$, which contains the cyclic subgroup $\langle \bar h \rangle$ as a finite-index subgroup. By definition, $EC_G(g)=\theta^{-1}(EC_{\bar G}(\bar g))$. 
    Thus, the preimage $\theta^{-1}(\langle \bar h \rangle)=K\cdot \langle h \rangle$ is a finite-index subgroup of $EC_G(g)$. 

    On the other hand, any element in $G$ that commutes with $h$ has an image in $\bar G$ that commutes with $\bar h$. Thus, $C_G(h)< EC_G(g)$. Note that $K\cdot \langle h \rangle<C_G(h)$ since $[H,K]=1$. In conclusion, $C_G(h)$ is a finite-index subgroup of $EC_G(g)$. 
\end{proof}

\begin{defn}
    For a relative HHG $(G,\s)$, we say a domain $V\in\s$ has \emph{neat kernel} if there exists a subgroup $H_V<\stab_G(V)$ such that $[H_V,\ker(\theta_V)]=1$ and $\theta_V$ maps $H_V$ to a finite-index subgroup of $G_V$. 
\end{defn}

\begin{prop}
\label{hhg_sep}
    Let $(G,\s)$ be a relative HHG that is residually finite. Let $V\in \s$. If $V$ has hyperbolicity, acylindrical image and neat kernel, then $V$ has separable quasi-axes. 
\end{prop}

\begin{proof}
    Let $g\in\stab_G(V)$ be a loxodromic element. Since $V$ has neat kernel, $EC(g)$ is a virtual centralizer in $\stab_G(V)$ of some loxodromic element $h\in H_V$ by Lemma \ref{centralizer}. Every element that commutes with $h$ stabilizes $V$. Therefore, $EC(g)$ is a virtual centralizer of $h$ in $G$. It is known that a centralizer in a residually finite group is separable (see \cite[Lemma 2.1]{BBF21} or the proof of \cite[Proposition]{Lon87}). It is also known that a finite-index supergroup of a separable subgroup is still separable (easy to see from the profinite topology). Therefore, $V$ has separable quasi-axes. 
\end{proof}

\begin{cor}
\label{cor_maxdomain}
    Let $(G,\s)$ be a relative HHG that is residually finite. Let $S\in \s$ be the unique maximal domain. Then $S$ is of type I.
\end{cor}

\begin{proof}
    By \cite[Theorem 14.3]{BHS17b}, $G$ acts on $\C S$ acylindrically. Now $\ker(\theta_V)$ is trivial so $S$ has neat kernel. Thus, $S$ has separable quasi-axes by Proposition \ref{hhg_sep}. Moreover, $S$ has cobounded nested region because $G$ acts on $\X$ coboundedly. In conclusion, $S$ is of type I. 
\end{proof}

\section{Applications}
\label{sec_apps}

\subsection{Mapping class groups}

In this subsection, we explain how Theorem \ref{maintheorem} applies to mapping class groups to recover the following theorem. 

\begin{thm}\cite[Theorem 1.2]{BBF21}
    Mapping class groups of finite-type surfaces have property QT.
\end{thm}

\begin{proof}
Let $\Sigma$ be a of finite-type surface, i.e. a closed oriented surface with finitely many marked points. Let $\M(\Sigma)$ be its marking complex \cite{MM00}. Let $\s$ be the collection of isotopy classes of essential non-pants closed subsurfaces of $\Sigma$, where disconnected subsurfaces are also allowed. Given any $V\in\s$, let $\hat V$ be the surface obtained by gluing a once-punctured disk to each boundary component of $V$. Let $\C V$ be the curve graph of $\hat V$. Here the curve graph of a disconnected surface is defined as the join of the curve graphs of its components, and thus is bounded. It is known that $(\M(\Sigma),\s)$ is an HHS. Moreover, the mapping class group $\mcg(\Sigma)$ is an HHG by virtue of its action on $(\M(\Sigma),\s)$ \cite[\S 11]{BHS19}. 

For any $V\in\s^{\circ}$, let $V^{\orth}$ be the closure of $\Sigma-V$ in $\Sigma$, and let $\mcg(\Sigma;V)<\mcg(\Sigma)$ be the stabilizer of $V$. Let $\eta_V:\mcg(V)\to \mcg(\Sigma)$ be the homomorphism induced by the inclusion $V\hookrightarrow \Sigma$. Denote the image of $\eta_V$ by $\overline\mcg(V)$. 
By \cite[Theorem 3.18]{FM12}, $\overline\mcg(V^{\orth})$ is exactly the kernel of the restriction homomorphism $\theta_V: \mcg(\Sigma;V)\to G_V$, where $G_V$ is a finite-index subgroup of $\mcg(\hat V)$. Moreover, it is clear that $\overline\mcg(V)$ commutes with $\overline\mcg(V^{\orth})$ and $\theta_V$ maps $\overline\mcg(V)$ to a finite-index subgroup of $G_V$. Therefore, $V$ has neat kernel.

It is known that $\mcg(\Sigma)$ is virtually colorable \cite[\S 5]{BBF15} and residually finite \cite{Gro74}. We only need to prove that every unbounded domain $V\in \s$ is of type I. First, it is clear that $\F_V$ is coarsely $\M(V)$. Since $\mcg(\hat V)$ acts coboundedly on $\M(V)$ and acts acylindrically on $\C \hat V$ \cite{Bow08}, $V$ has cobounded nested region and acylindrical image. Furthermore, $V$ has separable quasi-axes by the above discussion and Proposition \ref{hhg_sep}. Therefore, $V$ is of type I. 

In conclusion, mapping class groups of finite-type surfaces have property QT by Theorem \ref{maintheorem}. 
\end{proof}

Certain quotients of a mapping class group are again HHGs, as proved in \cite{BHS17a} and \cite{BHMS24}. In particular, the quotient by the normal closure of a suitable power of a pseudo-Anosov element or by the normal closure of suitable powers of all Dehn twists is again an HHG. It would be interesting to determine whether these quotient groups are still residually finite, thereby satisfying the assumption of Theorem \ref{maintheorem}.

\subsection{Admissible graphs of groups}

\emph{Admissible groups} were introduced by Croke--Kleiner in \cite{CK02}, which generalize the fundamental groups of non-geometric $3$-dimensional graph manifolds. 

\begin{defn}
\label{defn:admissible}
Let $\G=(\Gamma,\{G_v\},\{G_e\})$ be a graph of groups. We say $\mathcal{G}$ is \emph{admissible} if the following hold:
\begin{enumerate}
    \item $\Gamma$ is a finite graph with at least one edge.
    \item Each vertex group ${ G}_v$ has center $Z_v \cong \Z$, and ${ H}_v := { G}_{v} / Z_v$ is a non-elementary hyperbolic group.
    \item Every edge group ${ G}_{e}$ is isomorphic to $\Z^2$.
    \item If $e$ is an edge with $v=e^+$ and $w=e^-$, and $\tau_{e},\tau_{\bar e}$ are the edge monomorphisms, then the subgroup $\langle \tau_{e}^{-1}(Z_{v}),\tau_{\bar e}^{-1}(Z_{w})\rangle$ has finite index in ${ G}_e$.
    \item Let $e_1$ and $e_2$ be distinct edges entering a vertex $v$, and let $K_i \subset { G}_v$ be the image of the edge homomorphism $\tau_{e_i}$ for $i = 1,2$. Then
    \begin{itemize}
        \item for every $g \in G_v$, $gK_{1}g^{-1}$ is not commensurable with $K_2$;
        \item for every $g \in  G_v - K_i$, $gK_ig^{-1}$ is not commensurable with $K_i$. 
    \end{itemize} 
\end{enumerate}
A group $G$ is \emph{admissible} if it is the fundamental group of an admissible graph of groups.
\end{defn}

Every admissible group $G$ has a (combinatorial) HHG structure by \cite[Theorem 1.4]{HRSS25}. According to the classification of simplices by \cite[Lemma 6.2]{HRSS25}, if $\Delta\orth g\Delta$, where $\Delta$ corresponds to an unbounded hyperbolic space, then $\Delta$ is of type $8$ and $g$ exchanges two adjacent vertices in the Bass-Serre tree (see \cite[Definition 1.11]{BHMS24} for definition of orthogonality in a combinatorial HHS). Therefore, it is easy to see that $G$ has a subgroup of index at most $2$ that is colorable (see \cite[Lemma 4.6]{NY23a} for example). This shows the virtual colorability of $G$. Thus, every non-geometric graph manifold group has a virtually colorable HHG structure with all associated hyperbolic spaces being quasi-trees. Thus, non-geometric graph manifold groups have property QT by \cite[Theorem 3.1]{HP22} or Theorem \ref{maintheorem}. However, in the HHG structure of an admissible group, associated hyperbolic spaces are not necessarily quasi-trees. As an application of Theorem \ref{maintheorem}, we show that property QT still holds true in this case if we assume $G$ to be residually finite. 

\begin{thm}\label{admissibleQT}
    Let $\G=(\Gamma,\{G_v\},\{G_e\})$ be an admissible graph of groups, and let $G=\pi_1\G$. If $G$ is residually finite, then $G$ has property QT.
\end{thm}
\begin{proof}
    According to the classification of simplices \cite[Lemma 6.2]{HRSS25}, any simplex that is not of type $7$ corresponds to a quasi-tree so it is a domain of type II. Thus, we only need to check that simplices of type $7$ are of type I. The stabilizer of such a simplex $\Delta$ is exactly a vertex group $G_v$ that acts on $\C(\Delta)$ with image $H_v=G_v/Z_v$. Now $\C(\Delta)$ is coarsely the hyperbolic space obtained by coning off $H_v$ as a relatively hyperbolic group, and $\F_{\Delta}$ is coarsely $H_v$ itself. Therefore, acylindrical image and cobounded nested region hold true (see \cite[Proposition 5.2]{Osi16} for acylindricity). Since $G_v$ is a central extension of $H_v$ by $Z_v$, $\Delta$ has neat kernel. Therefore, we conclude by Proposition \ref{hhg_sep} and Theorem \ref{maintheorem}. 
\end{proof}

There is another approach to property QT of non-geometric graph manifold groups in \cite{HNY25}. For graph manifolds with nonempty boundary, they actually prove in a more general setting. A \emph{Croke--Kleiner admissible group} (abbreviated as CKA group) is an admissible group that admits a geometric action on a complete proper CAT(0) space. As a corollary of Theorem \ref{admissibleQT}, we recover the following theorem. 

\begin{cor}\cite[Theorem 1.3]{HNY25}\label{CKA}
    Let $G$ be a CKA group where for every vertex $v$ the central extension $1\to Z_v\to G_v\to H_v\to 1$ has an omnipotent hyperbolic quotient group $H_v$. Then $G$ has property QT.
\end{cor}

For definition of omnipotence, we refer the reader to \cite{Wis00}. Note that if every hyperbolic group is residually finite, then every hyperbolic group is omnipotent by \cite[Remark 3.4]{Wis00}. 
Under the assumption of Corollary \ref{CKA}, the central extension associated with any vertex virtually splits by \cite[Theorem II.7.1]{BH99}. Therefore, Theorem \ref{admissibleQT} implies Corollary \ref{CKA} due to the following lemma. 

\begin{lem}\label{CKARF}
    Let $G$ be an admissible group where for every vertex $v$ the central extension $1\to Z_v\to G_v\to H_v\to 1$ virtually splits and the hyperbolic quotient group $H_v$ is omnipotent. Then $G$ is residually finite. 
\end{lem}

We omit the proof of Lemma \ref{CKARF} since it is almost the same as the proof of residual finiteness for graph manifold groups by Hempel \cite{Hem87}. The reader can also see \cite{Ngu26} for an improved result.

\subsection{Hyperbolic-$2$-decomposable groups}

We say a group $G$ is \emph{hyperbolic-$2$-decomposable} if $G$ splits as a graph of hyperbolic groups with $2$-ended edge groups. 

\begin{thm}
    Let $G$ be a residually finite hyperbolic-$2$-decomposable group. Then $G$ has property QT if and only if $G$ does not contain any distorted element. 
\end{thm}

\begin{proof}
    If $G$ has property QT, then $G$ does not contain any distorted element by \cite[Lemma 2.5]{HNY25}. 
    Now assume that $G$ does not contain any distorted element. Let $\s$ be the HHG structure of $G$ given by \cite{RS20}. By construction, there is no orthogonality in $\s$. Thus, $G$ is colorable. Let $U\in \s$. Then $\C U$ is either a quasi-tree or a hyperbolic space obtained by coning off a vertex group $G_v$ as a relatively hyperbolic group. In the former case, $U$ is of type II. Now we only need to consider the latter case. Similarly to Theorem \ref{admissibleQT}, we have $\stab_G(U)=G_v$ and $\F_U$ is coarsely $G_v$ itself. It is easy to see that acylindrical image, cobounded nested region and neat kernel hold true. Therefore, we conclude by Proposition \ref{hhg_sep} and Theorem \ref{maintheorem}. 
\end{proof}

Similarly to Lemma \ref{CKARF}, if $G$ is a hyperbolic-$2$-decomposable group without any distorted element such that every vertex group is omnipotent, then $G$ is residually finite (see \cite[\S 4]{Wis00}). 

\subsection{Artin groups and extensions of lattice Veech groups}
\label{ArtinVeech}

Let $G$ be either
\begin{itemize}
    \item an Artin group of large and hyperbolic type, or
    \item the $\pi_1(\Sigma)$-extension group of a lattice Veech group in the mapping class group $\mcg(\Sigma)$ of a closed surface $\Sigma$. 
\end{itemize}

As shown in \cite{HMS24} and \cite{DDLS24} respectively, $G$ is a virtually colorable HHG. Moreover, the associated hyperbolic spaces of $G$ are all quasi-trees except the maximal one. By Corollary \ref{cor_maxdomain} and Theorem \ref{maintheorem}, $G$ has property QT if $G$ is residually finite. Hence, we obtain that

\begin{thm}
    Every residually finite Artin group of large and hyperbolic type has property QT. 
\end{thm}

It is proved in \cite{Jan22} that any $3$-generator Artin groups with labels $\ge 4$ except for $(2m +1, 4, 4)$ for any $m\ge 2$ is residually finite. Therefore, any $3$-generator Artin group with labels $\ge 4$ except for $(2m +1, 4, 4)$ for any $m\ge 2$ has property QT. On the other hand, we ask

\begin{que}\label{que: VeechRF}
    When is the $\pi_1(\Sigma)$-extension group of a lattice Veech group residually finite?
\end{que}

\begin{rem*}\label{rem: VeechRF}
    The answer to Question \ref{que: VeechRF} is that the $\pi_1(\Sigma)$-extension group $G$ of a lattice Veech group is always residually finite. Indeed, since the $\pi_1(\Sigma)$-extension is along the Birman exact sequence (see \cite[Theorem 4.6]{FM12} for example), $G$ is a subgroup of the mapping class group $\mcg(\Sigma,p)$ of $\Sigma$ with a marked point $p$. Since $\mcg(\Sigma,p)$ is residually finite \cite{Gro74}, $G$ is also residually finite. This leads to the following theorem.
\end{rem*}

\begin{thmVeechQT}
\phantomsection\label{thm: VeechQT}
    Every $\pi_1(\Sigma)$-extension group of a lattice Veech group has property QT and the QT embedding is quasi-median. 
\end{thmVeechQT}

\subsection{Graph products}

\begin{defn}[Graph product]
Let $\Gamma$ be a finite simplicial graph with the vertex set $V(\Gamma)$ and the edge set $E(\Gamma)$. Each vertex $v \in V(\Gamma)$ is labeled by a group $G_{v}$. The \emph{graph product} $G_\Gamma$ is the group 
\[ G_{\Gamma} = \left.\left(\free_{v \in V(\Gamma)} G_{v}\right) \right/ \llangle [g_{v},g_{w}] \,\, \middle| \,\, g_{v} \in G_{v},\, g_{w} \in G_{w},\, \{v,w\} \in E(\Gamma) \rrangle. \]
We call the $G_v$ the \emph{vertex groups} of the graph product $G_\Gamma$. 
\end{defn}

\begin{thm}\label{thm_graphproduct}
    Let $G_{\Gamma}$ be a graph product of groups whose every vertex group has property QT$_0$. Then $G_{\Gamma}$ has property QT$_0$.
\end{thm}

\begin{proof}
    Any graph product $G_\Gamma$ has a relative HHG structure $\s_\Gamma$ by \cite{BR22}. By definition of $\s_\Gamma$, any $G_\Gamma$-orbit on $\s_\Gamma$ corresponds to a unique subgraph of $\Gamma$ and is pairwise transversal. Thus, $G_\Gamma$ is colorable. By \cite[Theorem 4.4]{BR22}, for each domain $[g\Lambda]\in\s_\Gamma$, either $[g\Lambda]$ is $\nest$-minimal or $\C g\Lambda$ is a quasi-tree. Since each $\nest$-minimal domain corresponds to a vertex group, this means that every domain is of type II. By Theorem \ref{typeIIthm}, $G_{\Gamma}$ has property QT$_0$. 
\end{proof}

\printbibliography

@Article{BBF15,
  author       = {Bestvina, Mladen and Bromberg, Ken and Fujiwara, Koji},
  date         = {2015},
  journaltitle = {Publ. Math. Inst. Hautes Études Sci.},
  title        = {Constructing group actions on quasi-trees and applications to mapping class groups},
  doi          = {10.1007/s10240-014-0067-4},
  issn         = {0073-8301},
  pages        = {1--64},
  volume       = {122},
  journal      = {Publ. Math. Inst. Hautes Études Sci.},
  mrnumber     = {3415065},
}

@Article{BBF21,
  author       = {Bestvina, M. and Bromberg, K. and Fujiwara, K.},
  date         = {2021},
  journaltitle = {Annales Henri Lebesgue},
  title        = {Proper actions on finite products of quasi-trees},
  doi          = {10.5802/ahl.85},
  pages        = {685--709},
  volume       = {4},
  journal      = {Ann. H. Lebesgue},
  mrnumber     = {4315766},
}

@Article{BBFS19,
  author       = {Bestvina, Mladen and Bromberg, Ken and Fujiwara, Koji and Sisto, Alessandro},
  date         = {2019},
  journaltitle = {L'Enseignement Mathématique},
  title        = {Acylindrical actions on projection complexes},
  doi          = {10.4171/lem/65-1/2-1},
  issn         = {0013-8584},
  number       = {1-2},
  pages        = {1--32},
  volume       = {65},
  journal      = {L'Enseignement Mathématique},
  mrnumber     = {4057354},
}

@Article{BFG24,
  author       = {Balasubramanya, Sahana H. and Fournier-Facio, Francesco and Genevois, Anthony},
  title        = {Property ({NL}) for group actions on hyperbolic spaces (with an appendix by Alessandro Sisto)},
  issn         = {1661-7207},
  note         = {published online first},
  date         = {2024},
  doi          = {10.4171/ggd/806},
  journaltitle = {Groups, Geometry, and Dynamics},
  langid       = {english},
}

@Book{BH99,
  author    = {Bridson, Martin R. and Haefliger, André},
  date      = {1999},
  title     = {Metric spaces of non-positive curvature},
  doi       = {10.1007/978-3-662-12494-9},
  isbn      = {9783540643241},
  publisher = {Springer-Verlag, Berlin},
  series    = {Grundlehren der {Mathematischen} {Wissenschaften} [{Fundamental} {Principles} of {Mathematical} {Sciences}]},
  volume    = {319},
  mrnumber  = {1744486},
}

@Article{BHMS24,
  author       = {Behrstock, Jason and Hagen, Mark and Martin, Alexandre and Sisto, Alessandro},
  date         = {2024},
  journaltitle = {Journal of Topology},
  title        = {A combinatorial take on hierarchical hyperbolicity and applications to quotients of mapping class groups},
  doi          = {10.1112/topo.12351},
  issn         = {1753-8416,1753-8424},
  number       = {3},
  pages        = {Paper No. e12351, 94},
  volume       = {17},
  mrnumber     = {4822919},
  shortjournal = {J. Topol.},
}

@Article{BHS17a,
  author       = {Behrstock, Jason and Hagen, Mark F. and Sisto, Alessandro},
  date         = {2017},
  journaltitle = {Proceedings of the London Mathematical Society. Third Series},
  title        = {Asymptotic dimension and small-cancellation for hierarchically hyperbolic spaces and groups},
  doi          = {10.1112/plms.12026},
  issn         = {0024-6115},
  number       = {5},
  pages        = {890--926},
  volume       = {114},
  mrnumber     = {3653249},
}

@Article{BHS17b,
  author       = {Behrstock, Jason and Hagen, Mark F. and Sisto, Alessandro},
  date         = {2017},
  journaltitle = {Geom. Topol.},
  title        = {Hierarchically hyperbolic spaces, {I}: {Curve} complexes for cubical groups},
  doi          = {10.2140/gt.2017.21.1731},
  issn         = {1465-3060},
  number       = {3},
  pages        = {1731--1804},
  volume       = {21},
  journal      = {Geom. Topol.},
  mrnumber     = {3650081},
  shorttitle   = {Hierarchically hyperbolic spaces, {I}},
}

@Article{BHS19,
  author       = {Behrstock, J. and Hagen, M. and Sisto, A.},
  date         = {2019},
  journaltitle = {Pacific J. Math.},
  title        = {Hierarchically hyperbolic spaces {II}: {Combination} theorems and the distance formula},
  doi          = {10.2140/pjm.2019.299.257},
  issn         = {0030-8730},
  number       = {2},
  pages        = {257--338},
  volume       = {299},
  journal      = {Pac. J. Math.},
  mrnumber     = {3956144},
  shorttitle   = {Hierarchically hyperbolic spaces {II}},
}

@Article{Bow08,
  author       = {Bowditch, Brian H.},
  date         = {2008},
  journaltitle = {Inventiones Mathematicae},
  title        = {Tight geodesics in the curve complex},
  doi          = {10.1007/s00222-007-0081-y},
  issn         = {0020-9910},
  number       = {2},
  pages        = {281--300},
  volume       = {171},
  mrnumber     = {2367021},
}

@Article{BR22,
  author       = {Berlyne, Daniel and Russell, Jacob},
  date         = {2022},
  journaltitle = {Groups, Geometry, and Dynamics},
  title        = {Hierarchical hyperbolicity of graph products},
  doi          = {10.4171/ggd/652},
  issn         = {1661-7207},
  number       = {2},
  pages        = {523--580},
  volume       = {16},
  langid       = {english},
}

@Article{But25,
  author       = {Button, Jack O.},
  title        = {Generalised Baumslag-Solitar groups and hierarchically hyperbolic groups},
  issn         = {1472-2747,1472-2739},
  number       = {4},
  pages        = {2253--2279},
  volume       = {25},
  date         = {2025},
  doi          = {10.2140/agt.2025.25.2253},
  journaltitle = {Algebraic \& Geometric Topology},
  mrnumber     = {4950904},
  shortjournal = {Algebr. Geom. Topol.},
}

@Article{Che22,
  author       = {Chesser, Marissa},
  date         = {2022},
  journaltitle = {Algebraic \& Geometric Topology},
  title        = {Stable subgroups of the genus 2 handlebody group},
  doi          = {10.2140/agt.2022.22.919},
  issn         = {1472-2747,1472-2739},
  number       = {2},
  pages        = {919--971},
  volume       = {22},
  mrnumber     = {4464468},
  shortjournal = {Algebr. Geom. Topol.},
}

@Article{CK02,
  author       = {Croke, C. B. and Kleiner, B.},
  date         = {2002},
  journaltitle = {Geometric and Functional Analysis},
  title        = {The geodesic flow of a nonpositively curved graph manifold},
  doi          = {10.1007/s00039-002-8255-7},
  issn         = {1016-443X},
  number       = {3},
  pages        = {479--545},
  volume       = {12},
  mrnumber     = {1924370},
}

@Article{DDLS24,
  author       = {Dowdall, Spencer and Durham, Matthew G. and Leininger, Christopher J. and Sisto, Alessandro},
  title        = {Extensions of Veech groups {II}: Hierarchical hyperbolicity and quasi-isometric rigidity},
  issn         = {0010-2571,1420-8946},
  number       = {1},
  pages        = {149--228},
  volume       = {99},
  date         = {2024},
  doi          = {10.4171/cmh/568},
  journaltitle = {Commentarii Mathematici Helvetici. A Journal of the Swiss Mathematical Society},
  mrnumber     = {4709309},
  shortjournal = {Comment. Math. Helv.},
  shorttitle   = {Extensions of Veech groups {II}},
}

@Book{DGO17,
  author    = {Dahmani, F. and Guirardel, V. and Osin, D.},
  date      = {2017},
  title     = {Hyperbolically embedded subgroups and rotating families in groups acting on hyperbolic spaces},
  doi       = {10.1090/memo/1156},
  isbn      = {9781470436018;9781470421946},
  language  = {en},
  number    = {1156},
  publisher = {American Mathematical Society},
  series    = {Memoirs of the {American} {Mathematical} {Society}},
  volume    = {245},
  issn      = {0065-9266, 1947-6221},
}

@Article{DHS20,
  author       = {Durham, Matthew and Hagen, Mark and Sisto, Alessandro},
  date         = {2020},
  journaltitle = {Geometry \& Topology},
  title        = {Correction to the article {Boundaries} and automorphisms of hierarchically hyperbolic spaces},
  doi          = {10.2140/gt.2020.24.1051},
  issn         = {1364-0380, 1465-3060},
  language     = {en},
  number       = {2},
  pages        = {1051--1073},
  volume       = {24},
  journal      = {Geometry \& Topology},
}

@InProceedings{DJ99,
  author    = {Dranishnikov, A. and Januszkiewicz, T.},
  booktitle = {Topology {Proceedings}},
  date      = {1999},
  title     = {Every {Coxeter} group acts amenably on a compact space},
  number    = {Spring},
  pages     = {135--141},
  volume    = {24},
  issn      = {0146-4124},
  mrnumber  = {1802681},
}

@Article{DMS23,
  author       = {Durham, Matthew G. and Minsky, Yair N. and Sisto, Alessandro},
  date         = {2023},
  journaltitle = {Geometry \& Topology},
  title        = {Stable cubulations, bicombings, and barycenters},
  doi          = {10.2140/gt.2023.27.2383},
  issn         = {1465-3060,1364-0380},
  number       = {6},
  pages        = {2383--2478},
  volume       = {27},
  mrnumber     = {4634751},
  shortjournal = {Geom. Topol.},
}

@Book{FM12,
  author    = {Farb, Benson and Margalit, Dan},
  date      = {2012},
  title     = {A primer on mapping class groups},
  isbn      = {9780691147949},
  publisher = {Princeton University Press, Princeton, NJ},
  series    = {Princeton {Mathematical} {Series}},
  volume    = {49},
  mrnumber  = {2850125},
}

@Article{Gro74,
  author       = {Grossman, Edna K.},
  date         = {1974},
  journaltitle = {Journal of the London Mathematical Society. Second Series},
  title        = {On the residual finiteness of certain mapping class groups},
  doi          = {10.1112/jlms/s2-9.1.160},
  issn         = {0024-6107},
  pages        = {160--164},
  volume       = {9},
  mrnumber     = {405423},
}

@InCollection{Gro87,
  author    = {Gromov, M.},
  booktitle = {Essays in group theory},
  date      = {1987},
  title     = {Hyperbolic groups},
  doi       = {10.1007/978-1-4613-9586-7_3},
  pages     = {75--263},
  publisher = {Springer},
  series    = {Math. {Sci}. {Res}. {Inst}. {Publ}.},
  volume    = {8},
  mrnumber  = {919829},
}

@Article{Hag23,
  author       = {Hagen, Mark},
  date         = {2023},
  journaltitle = {International Journal of Algebra and Computation},
  title        = {Non-colorable hierarchically hyperbolic groups},
  doi          = {10.1142/S0218196723500170},
  issn         = {0218-1967},
  number       = {02},
  pages        = {337--350},
  volume       = {33},
  publisher    = {World Scientific Publishing Co.},
}

@Article{Hem87,
  author    = {Hempel, John},
  date      = {1987},
  title     = {Residual finiteness for 3-manifolds},
  doi       = {10.1515/9781400882083-018},
  pages     = {379--396},
  series    = {Ann. of Math. Stud.},
  volume    = {111},
  booktitle = {Combinatorial group theory and topology (Alta, Utah, 1984)},
  isbn      = {9780691084091 9780691084107},
  mrnumber  = {895623},
  publisher = {Princeton Univ. Press, Princeton, {NJ}},
}

@Article{HMS24,
  author       = {Hagen, Mark and Martin, Alexandre and Sisto, Alessandro},
  date         = {2024},
  journaltitle = {Mathematische Annalen},
  title        = {Extra-large type Artin groups are hierarchically hyperbolic},
  doi          = {10.1007/s00208-022-02523-4},
  issn         = {0025-5831,1432-1807},
  number       = {1},
  pages        = {867--938},
  volume       = {388},
  mrnumber     = {4693949},
  shortjournal = {Math. Ann.},
}

@Article{HNY25,
  author       = {Han, Suzhen and Nguyen, Hoang Thanh and Yang, Wenyuan},
  date         = {2025},
  journaltitle = {Algebraic \& Geometric Topology},
  title        = {Property ({QT}) for 3-manifold groups},
  doi          = {10.2140/agt.2025.25.107},
  issn         = {1472-2747,1472-2739},
  number       = {1},
  pages        = {107--159},
  volume       = {25},
  mrnumber     = {4877252},
  shortjournal = {Algebr. Geom. Topol.},
}

@Article{HP22,
  author       = {Hagen, Mark F. and Petyt, Harry},
  date         = {2022},
  journaltitle = {Algebraic \& Geometric Topology},
  title        = {Projection complexes and quasimedian maps},
  doi          = {10.2140/agt.2022.22.3277},
  issn         = {1472-2739},
  language     = {en},
  number       = {7},
  pages        = {3277--3304},
  volume       = {22},
  journal      = {Algebraic \& Geometric Topology},
  publisher    = {Mathematical Sciences Publishers},
}

@Article{HRSS25,
  author       = {Hagen, Mark and Russell, Jacob and Sisto, Alessandro and Spriano, Davide},
  title        = {Equivariant hierarchically hyperbolic structures for 3-manifold groups via quasimorphisms},
  issn         = {0373-0956,1777-5310},
  number       = {2},
  pages        = {769--828},
  volume       = {75},
  date         = {2025},
  doi          = {10.5802/aif.3654},
  fjournal     = {Universit\'e{} de Grenoble. Annales de l'Institut Fourier},
  journaltitle = {Ann. Inst. Fourier (Grenoble)},
  mrclass      = {20F67 (57K35)},
  mrnumber     = {4921365},
  mrreviewer   = {Bin\ Sun},
  url          = {https://doi.org/10.5802/aif.3654},
}

@Article{Jan22,
  author       = {Jankiewicz, Kasia},
  date         = {2022},
  journaltitle = {Advances in Mathematics},
  title        = {Residual finiteness of certain 2-dimensional Artin groups},
  doi          = {10.1016/j.aim.2022.108487},
  issn         = {0001-8708,1090-2082},
  pages        = {Paper No. 108487, 37},
  volume       = {405},
  mrnumber     = {4437605},
  shortjournal = {Adv. Math.},
}

@Article{Lon87,
  author       = {Long, D. D.},
  date         = {1987},
  journaltitle = {The Bulletin of the London Mathematical Society},
  title        = {Immersions and embeddings of totally geodesic surfaces},
  doi          = {10.1112/blms/19.5.481},
  issn         = {0024-6093,1469-2120},
  number       = {5},
  pages        = {481--484},
  volume       = {19},
  mrnumber     = {898729},
  shortjournal = {Bull. London Math. Soc.},
}

@Article{Man06,
  author       = {Manning, J. F.},
  date         = {2006},
  journaltitle = {Journal of the London Mathematical Society. Second Series},
  title        = {Quasi-actions on trees and property ({QFA})},
  doi          = {10.1112/S0024610705022738},
  issn         = {0024-6107,1469-7750},
  number       = {1},
  pages        = {84--108},
  volume       = {73},
  mrnumber     = {2197372},
  shortjournal = {J. London Math. Soc. (2)},
}

@Article{Man08,
  author       = {Manning, Jason Fox},
  date         = {2008},
  journaltitle = {Algebraic \& Geometric Topology},
  title        = {Actions of certain arithmetic groups on Gromov hyperbolic spaces},
  doi          = {10.2140/agt.2008.8.1371},
  issn         = {1472-2747,1472-2739},
  number       = {3},
  pages        = {1371--1402},
  volume       = {8},
  mrnumber     = {2443247},
  shortjournal = {Algebr. Geom. Topol.},
}

@Article{MM00,
  author       = {Masur, H. A. and Minsky, Y. N.},
  date         = {2000},
  journaltitle = {Geometric and Functional Analysis},
  title        = {Geometry of the complex of curves. {II}. {Hierarchical} structure},
  doi          = {10.1007/PL00001643},
  issn         = {1016-443X},
  number       = {4},
  pages        = {902--974},
  volume       = {10},
  journal      = {Geom. Funct. Anal.},
  mrnumber     = {1791145},
}

@Article{MM99,
  author       = {Masur, H. A. and Minsky, Y. N.},
  date         = {1999},
  journaltitle = {Inventiones Mathematicae},
  title        = {Geometry of the complex of curves. {I}. {Hyperbolicity}},
  doi          = {10.1007/s002220050343},
  issn         = {0020-9910},
  number       = {1},
  pages        = {103--149},
  volume       = {138},
  journal      = {Invent. Math.},
  mrnumber     = {1714338},
}

@Article{Ngu26,
  author       = {Nguyen, Hoang Thanh},
  title        = {Separability properties of extended admissible groups},
  issn         = {0933-7741,1435-5337},
  number       = {2},
  pages        = {519--537},
  volume       = {38},
  date         = {2026},
  doi          = {10.1515/forum-2024-0194},
  fjournal     = {Forum Mathematicum},
  journaltitle = {Forum Math.},
  mrclass      = {20F65 (20F67)},
  mrnumber     = {5009101},
  url          = {https://doi.org/10.1515/forum-2024-0194},
}

@Article{NY23a,
  author       = {Nguyen, Hoang Thanh and Yang, Wenyuan},
  date         = {2023},
  journaltitle = {Michigan Mathematical Journal},
  title        = {Croke-Kleiner admissible groups: property ({QT}) and quasiconvexity},
  doi          = {10.1307/mmj/20216045},
  issn         = {0026-2285,1945-2365},
  number       = {5},
  pages        = {971--1019},
  volume       = {73},
  mrnumber     = {4665613},
  shortjournal = {Michigan Math. J.},
  shorttitle   = {Croke-Kleiner admissible groups},
}

@Article{Osi16,
  author       = {Osin, D.},
  date         = {2016},
  journaltitle = {Transactions of the American Mathematical Society},
  title        = {Acylindrically hyperbolic groups},
  doi          = {10.1090/tran/6343},
  issn         = {0002-9947},
  number       = {2},
  pages        = {851--888},
  volume       = {368},
  mrnumber     = {3430352},
}

@Article{Pet21,
  author     = {Petyt, Harry},
  date       = {2021},
  title      = {Mapping class groups are quasicubical},
  doi        = {10.48550/arXiv.2112.10681},
  note       = {arXiv:2112.10681 [math]},
}

@Article{PS23,
  author       = {Petyt, Harry and Spriano, Davide},
  date         = {2023},
  journaltitle = {Groups, Geometry, and Dynamics},
  title        = {Unbounded domains in hierarchically hyperbolic groups},
  doi          = {10.4171/ggd/706},
  issn         = {1661-7207},
  language     = {en},
}

@Article{PSZ25,
  author     = {Petyt, Harry and Spriano, Davide and Zalloum, Abdul},
  date       = {2025},
  title      = {Stable cylinders and fine structures for hyperbolic groups and curve graphs},
  doi        = {10.48550/arXiv.2501.13600},
  eprinttype = {arxiv},
  number     = {{arXiv}:2501.13600},
  publisher  = {{arXiv}},
}

@Article{RS95,
  author       = {Rips, E. and Sela, Z.},
  date         = {1995},
  journaltitle = {Inventiones Mathematicae},
  title        = {Canonical representatives and equations in hyperbolic groups},
  doi          = {10.1007/BF01241140},
  issn         = {0020-9910,1432-1297},
  number       = {3},
  pages        = {489--512},
  volume       = {120},
  mrnumber     = {1334482},
  shortjournal = {Invent. Math.},
}

@Article{RS20,
  author   = {Robbio, Bruno and Spriano, Davide},
  title    = {Hierarchical hyperbolicity of hyperbolic-2-decomposable groups},
  note     = {arXiv:2007.13383 [math]},
  date     = {2020},
  doi      = {10.48550/arXiv.2007.13383},
}

@Article{Ver24,
  author    = {Vergara, Ignacio},
  title     = {Quasi-trees, Lipschitz free spaces, and actions on \${\textbackslash}ell{\textasciicircum}1\$},
  number    = {{arXiv}:2409.14186},
  date      = {2024},
  doi       = {10.48550/arXiv.2409.14186},
  publisher = {{arXiv}},
}

@Article{Wis00,
  author       = {Wise, Daniel T.},
  date         = {2000},
  journaltitle = {The Quarterly Journal of Mathematics},
  title        = {Subgroup separability of graphs of free groups with cyclic edge groups},
  doi          = {10.1093/qmathj/50.1.107},
  issn         = {0033-5606,1464-3847},
  number       = {1},
  pages        = {107--129},
  volume       = {51},
  mrnumber     = {1760573},
  shortjournal = {Q. J. Math.},
}

@Article{Yan19,
  author       = {Yang, Wen-yuan},
  date         = {2019},
  journaltitle = {International Mathematics Research Notices. IMRN},
  title        = {Statistically convex-cocompact actions of groups with contracting elements},
  doi          = {10.1093/imrn/rny001},
  issn         = {1073-7928},
  number       = {23},
  pages        = {7259--7323},
  mrnumber     = {4039013},
}
\end{document}